\documentclass{article}
\usepackage[margin=1in]{geometry}
\usepackage{amsmath}
\usepackage{amssymb}
\usepackage{amsthm}
\usepackage{amscd}
\usepackage{enumerate}
\usepackage[colorlinks=true]{hyperref}
\usepackage{bbm}
\usepackage{etoolbox}

\usepackage[utf8]{inputenc}

\newcommand{\tomemail}{\href{mailto:tom.bachmann@zoho.com}{tom.bachmann@zoho.com}}

\ifdefvoid{\longdocument}{
\newtheorem*{defn}{Definition}
\newtheorem{prop}{Proposition}
}{
\newtheorem{prop}{Proposition}[chapter]

}
\newtheorem{corr}[prop]{Corollary}
\newtheorem{lemm}[prop]{Lemma}
\newtheorem{thm}[prop]{Theorem}

\newtheorem*{thm*}{Theorem}
\newtheorem*{corr*}{Corollary}
\newtheorem*{prop*}{Proposition}
\newtheorem*{lemm*}{Lemma}
\theoremstyle{definition}

\theoremstyle{remark}

\newcommand{\gmtw}[1]{\langle #1 \rangle}
\newcommand{\Aone}{{\mathbb{A}^1}}
\newcommand{\Gm}{{\mathbb{G}_m}}

\newcommand{\id}{\operatorname{id}}
\newcommand{\ZZ}{\mathbb{Z}}

\newcommand{\NN}{\mathbb{N}}

\newcommand{\DM}{\mathbf{DM}}
\newcommand{\HI}{\mathbf{HI}}

\newcommand{\SH}{\mathbf{SH}}

\newcommand{\iHom}{\ul{\operatorname{Hom}}}
\newcommand{\Hom}{\operatorname{Hom}}

\newcommand{\colim}{\operatorname{colim}}
\newcommand{\hocolim}{\operatorname{hocolim}}

\newcommand{\RR}{\mathbb{R}}

\newcommand{\Mod}{\textbf{-Mod}}

\DeclareRobustCommand{\ul}{\underline}
\newcommand{\heart}{\heartsuit}

\newcommand{\tunit}{\mathbbm{1}}
\newcommand{\ret}{{r\acute{e}t}}

\newcommand{\iso}{\cong}
\newcommand{\wequi}{\simeq}

\usepackage{tikz-cd}
\def\Aone{\mathbb{A}^1}

\title{On the Conservativity of the Functor Assigning to a Motivic Spectrum its
       Motive}
\author{Tom Bachmann \\ \tomemail}

\begin{document}
\maketitle

\begin{abstract}
Given a $0$-connective motivic spectrum $E \in \SH(k)$ over a perfect
field $k$, we determine $\ul h_0$ of the associated motive $ME \in \DM(k)$ in
terms of $\ul \pi_0(E)$.
Using this we show that if $k$ has finite $2$-étale cohomological dimension,
then the functor $M: \SH(k) \to \DM(k)$ 
is conservative when restricted to the subcategory of compact spectra,
and induces an injection on Picard groups. We extend the conservativity result
to fields of finite \emph{virtual} 2-étale cohomological dimension by
considering what we call ``real motives''.
\end{abstract}

\section{Introduction}
One of the starting points of classical stable homotopy theory is the \emph{Hurewicz
theorem}: we call a spectrum $E \in \SH$ \emph{0-connective} if $\pi_i (E) = 0$ for all
$i < 0$, where $\pi_i$ denote the stable homotopy groups. We may also associate
with $E$ its homology groups $H_i(E) := H_i(C_*E)$, where $C_* E$ denotes the
singular chain complex. The Hurewicz Theorem states
that if $E$ is 0-connective then $H_i(E) = 0$ for all $i < 0$ and that there is
a natural isomorphism $\pi_0(E) \iso H_0(E)$.

This has several nice consequences. For example, if $E
\in \SH$ is a connective spectrum (has only finitely many non-vanishing negative
stable homotopy groups) and $C_* E \wequi 0$, then actually $E \simeq 0$. We
call this the \emph{conservativity theorem}.

For another example, recall that a spectrum $E$ is called \emph{invertible} if there exists
a spectrum $F$ such that $E \wedge F \wequi S$ ($S$ the sphere spectrum). The
set of equivalence classes of invertible spectra forms an abelian group denoted
$Pic(\SH)$; similarly we have the group $Pic(D(Ab))$. The functor $C_*$ induces
a homomorphism $Pic(\SH) \to Pic(D(Ab))$ which is \emph{injective}. (It follows
then quickly that it is an isomorphism, since $Pic(D(Ab))$ is easy to
determine.) We call this the \emph{Pic-injectivity theorem}.

We wish to investigate how to extend these results to the motivic world. Recall
that in this world, for every field $k$ (at least) there is defined a category
$\SH(k)$ analogous to $\SH$, but starting with smooth varieties as building
blocks instead of topological spaces. In complete analogy with the classical
situation, Morel, Ayoub and others have
defined and studied a category $D_\Aone(k)$ and a functor $C_*: \SH(k) \to
D_\Aone(k)$; see for example \cite[Section 5.2]{morel2004motivic-pi0}. Both of the
categories $\SH(k)$ and $D_\Aone(k)$ carry a $t$-structure (known as the
\emph{homotopy $t$-structure}, not the conjectural motivic $t$-structure),
and $C_*$ induces
an isomorphisms of the hearts. (See Section \ref{sec:abstract-hurewicz} for our
conventions regarding $t$-structures.) Let us write $\ul \pi_i(E)$ for the homotopy
objects of a spectrum $E \in \SH(k)$ in this $t$-structure, and $\ul h'_i(E)$ for
the homotopy objects of $C_* E$. Then Morel has proved that for 0-connective $E$
one has $\ul \pi_0(E) \iso \ul h'_0(E)$ \cite[Theorem 6.37]{A1-alg-top}. This is a
perfect analog of the classical Hurewicz theorem, and the conservativity and
Pic-injectivity theorems follow in the same way as classically.

The problem with this formulation is that the category $D_\Aone(k)$ is not yet
very well understood. Moreover Voevodsky has constructed the category $\DM(k)$
which some may consider to be a more accessible target category for a Hurewicz
Theorem. But even
though $\DM(k)$ has a $t$-structure, and there is a nice functor $M: \SH(k)
\to \DM(k)$, the hearts of $\SH(k)$ and $\DM(k)$ are definitely \emph{not}
equivalent, so some new idea is needed.
Write $\ul h_i(E)$ for the homotopy objects of $ME$ (in the homotopy
$t$-structure on $\DM(k)$) for $E \in \SH(k)$.
Suppose $E \in \SH(k)$ is 0-connective. Our first contribution is to identify
$\ul h_0(E)$.

\begin{thm*}[Motivic Hurewicz Theorem; see Theorem \ref{thm:hurewicz}]
Let $E \in \SH(k)$ be 0-connective. Then in notation to be explained later, we
have
\[ \ul h_0(E) \iso \ul \pi_0(E)/\eta. \]
\end{thm*}
Here we use implicitly that the heart of $\DM(k)$ can be identified with a
subcategory of the heart of $\SH(k)$. This theorem shows that we lose some
information in passing to the motive, and so cannot expect a perfect analogy
with the classical situation. Nonetheless we can prove the following.

\begin{thm*}[Conservativity Theorem I; see Theorem \ref{thm:conservativity-I}]
Let $k$ be a perfect field of finite 2-étale cohomological dimension and $E \in
\SH(k)$ be compact. If $ME \wequi 0$ then $E \simeq 0$.
\end{thm*}
(Recall that an object is called compact if $\Hom_{\SH(k)}(E, \bullet)$ commutes
with arbitrary sums.) Combined with the Motivic Hurewicz Theorem, one easily
obtains the following.
\begin{corr*}[Pic-injectivity Theorem; see Theorem \ref{thm:pic-injectivity}]
In the situation of the Theorem, the natural homomorphism $Pic(\SH(k)) \to
Pic(\DM(k))$ is injective.
\end{corr*}

This result was one of the main motivations for our investigations. See later in
this introduction for more details on applications. (We do not know if $Pic(\SH(k)) \to
Pic(\DM(k))$ might also be surjective.)

The Conservativity Theorem (and by extension the Pic-injectivity Theorem)
stated here is not optimal. Let us discuss to what extent the
assumptions can be weakened.

In
characteristic zero, the compactness assumption can be replaced by the technical
assumptions of connectivity and ``slice-connectivity'' (see Section
\ref{sec:homotopy-modules} for a definition of this term). In characteristic $p>0$ we
can also replace compactness by connectivity and slice-connectivity, but in this
case we must additionally assume that the natural map $E \xrightarrow{p} E$ is
an isomorphism.

If $k$ is a non-orderable field (i.e. one in which -1 is a sum of squares; we
still assume $k$ perfect), then
one may present $k$ as a colimit of perfect subfields with finite 2-étale
cohomological dimension, and the theorem can then be deduced for such $k$,
see \cite{bondarko-effectivity}. Compactness instead of connectivity and
slice-connectivity is then essential, however.

The author currently does not know how to weaken the perfectness assumption on $k$.

A further weakening is to consider orderable $k$, i.e. fields
where -1 is \emph{not} a sum of squares. (Such fields necessarily have
characteristic zero.) Assume for simplicity that for any ordering of $k$, there
is an order preserving embedding of $k$ into $\RR$.
Write $Sper(k)$ for the set of orderings of $k$. (Equivalently in this case,
embeddings $k \hookrightarrow \RR$.) If $\sigma \in Sper(k)$, then there is a so-called
\emph{real realisation} functor $R_\sigma^\RR: \SH(k) \to \SH$, related to
considering the real points of a smooth variety, with their strong topology. We
compose this with the singular chain complex functor to obtain the \emph{real
motive} $M_\sigma^\RR: \SH(k) \to D(Ab)$. We can further consider homology with
$\ZZ[1/2]$ coefficients and obtain $M_\sigma^\RR[1/2]: \SH(k) \to D(\ZZ[1/2])$. We may
then prove the following.

\begin{thm*}[Conservativity Theorem II; see Theorem \ref{thm:conservativity-II}]
Let $k$ be a field of finite virtual 2-étale cohomological
dimension with real embeddings for all orderings, as above.
Let $E \in \SH(k)$ be connective and slice-connective. If $0
\wequi ME \in \DM(k)$ and additionally for all $\sigma \in Sper(k)$, $0 \simeq
M_\sigma^\RR[1/2](E) \in D(\ZZ[1/2])$, then in fact $E \wequi 0$.
\end{thm*}

We point out right away that this theorem does not include a Pic-injectivity
result; in the current form such a result seems unlikely to hold. The conditions
on $k$ can again be relaxed. It turns out that for any orderable field $k$ and
ordering $\sigma \in Sper(k)$, one may define a functor $M_\sigma^\RR[1/2]: \SH(k)
\to D(\ZZ[1/2])$ analogous to the real motive functors for $k \subset \RR$; we
still call them real motives. The theorem holds for all $k$ of characteristic
zero and finite virtual 2-étale cohomological dimension, with the new definition of
$M_\sigma^\RR[1/2]$. Details are explained before the proof of Theorem
\ref{thm:conservativity-II}. One may also show that any orderable field is a
colimit of orderable fields of finite virtual 2-étale cohomological dimension, so the
theorem holds for arbitrary orderable fields, assuming that $E$ is
\emph{compact}. See again Bondarko's work \cite{bondarko-effectivity}.

A word on applications. In the classical situation, Pic-injectivity immediately
implies that $Pic(\SH) = \ZZ.$ In the motivic case, results are not nearly as
strong, mainly because not much is known about $Pic(\DM(k)).$ However, combining
conservativity and Pic-injectivity, many problems about (potentially) invertible
objects in $\SH(k)$ can be reduced to analogous questions in $\DM(k)$. A
prototypical example is as follows (the restriction to $k$ of characteristic
zero can be overcome by inverting the exponential characteristic, we omit this
here to simplify exposition):

\begin{prop*} Let $E, F \in \SH(k)$ be compact and suppose that $k$ is of
finite 2-étale cohomological dimension and characteristic zero.
Then $E, F$ are isomorphic invertible objects if and only if $ME, MF$
are.
\end{prop*}
\begin{proof}
It is clear that if $E \wequi F$ are invertible, then so are $ME, MF.$ We show
the converse.

Since $M$ is a monoidal functor between rigid monoidal categories, it preserves
duals. Consider the natural map $DE \wedge E \xrightarrow{\alpha}
\tunit_{\SH(k)}.$ If $ME$ is invertible, then $M\alpha$ is an isomorphism, and
hence so is $\alpha,$ by conservativity. Thus $E$ is invertible. Similarly for
$F.$ But now $E, F \in Pic(\SH(k))$ and $E \wequi F$ if and only if $ME \wequi
MF,$ by Pic-injectivity.
\end{proof}

The conjectures of Po Hu on invertibility of Pfister quadrics, as stated in
\cite[Conjecture 1.4]{HuPicard}, can be cast in this form. We shall establish
the analog of Hu's conjecture in $\DM(k)$ in a forthcoming article \cite{bachmann-quadrics}.
Together with the results of this work, this establishes Hu's conjecture for $\SH(k)$
whenever $k$ has finite $2$-étale cohomological dimension.

If one only wishes to show that an object is invertible, we do not need
Pic-injectivity. For example, one idea is that the reduced suspension
spectrum of any smooth affine quadric $Q$ should yield an invertible object in
$\SH(k)$. Indeed this holds true under real and complex realisation, and also
follows for Pfister Quadrics from Po Hu's conjectures. Additionally it is true in
the étale topology.

By arguments as in the above proposition, using the two conservativity
theorems, it follows that it suffices to show that the reduced motive $\tilde{M}Q$
is invertible
and that for every ordering $\sigma$ of $k$, the reduced real motive
$\tilde{M}_\sigma^\RR[1/2] Q$ is
invertible. (For a variety $X/k$ we have the reduced suspension spectrum
$\tilde{\Sigma}^\infty X := hofib(\Sigma^\infty X_+ \to \Sigma^\infty
Spec(k)_+)$ and $\tilde{M}X := M(\tilde{\Sigma}^\infty X),
\tilde{M}_\sigma^\RR[1/2] X :=  M_\sigma^\RR[1/2](\tilde{\Sigma}^\infty X)$.)
The former question is purely about motives, and dealt with in the
same forthcoming work mentioned above. The latter question is purely topological
and very easy. Thus we conclude: if $k$ is a field of characteristic $0$ and $Q$ any
smooth affine quadric over $k$, then $\Sigma^\infty Q$ is invertible in
$\SH(k)$. (More generally, this holds for \emph{any} perfect field after
inverting the characteristic.)

Here is a detailed description of the organisation of this paper. In Section
\ref{sec:abstract-hurewicz} we prove a general result about $t$-categories and
certain functors between them, which we call the \emph{abstract Hurewicz
Theorem}. It will be our replacement for the classical Hurewicz Theorem.

In Section \ref{sec:motivic-homotopy-theory}, we recall the construction of
$\SH(k)$ and $\DM(k)$ with the appropriate $t$-structures. We show how to
combine the abstract Hurewicz Theorem with a fundamental result of Deglise to
deduce our motivic Hurewicz Theorem.

In Section \ref{sec:homotopy-modules} we study the heart of the category $\SH(k)$
in more detail. It is also known as the category of \emph{homotopy modules}.
We combine the careful study of Morel of strictly invariant
sheaves with Levine's work on Voevodsky's slice filtration to show that every
homotopy module $H$ has a canonical filtration $F_\bullet$ built using the
unramified Milnor-Witt K-theory sheaves $\ul K_*^{MW}$. Combined with
Voevodsky's resolution of the Milnor conjectures, we can prove the
conservativity theorem for $k$ of finite 2-étale cohomological dimension and
with the exponential characteristic inverted.

In order to remove the need to invert the exponential characteristic, one uses a
very similar (probably identical) filtration $\tilde{F}_\bullet H$. In this
section we only state its properties; the actual construction is relegated to an
appendix.

The Pic-injectivity theorem is obtained as an easy consequence of the
conservativity result.

The rest of the paper deals with the case of orderable fields. Using methods
very similar to the unorderable case, one shows: if $E$ is connective and
slice-connective and $ME \wequi 0$, then $2$ is invertible on $E$.
This is a crucial insight of Bondarko \cite{bondarko-effectivity}. (This is where we need finite virtual
2-étale cohomological dimension.)

In Section \ref{sec:witt-motives} we tackle the problem of building a
conservative functor for spectra on which $2$ is invertible. We consider a
category $\DM_W(k, \ZZ[1/2])$ with a functor $F: \SH(k) \to \DM_W(k, \ZZ[1/2])$
factoring through $\SH(k)_2^-$.
It has already been
studied in \cite{levine2015witt}. An easy argument using the abstract Hurewicz
theorem shows that if $E \in \SH(k)$ is connective and slice connective,
$ME \wequi 0$ and also $FE \simeq 0$,
then $E \wequi 0$. Consequently from this part on we concentrate on studying the
category $\DM_W(k, \ZZ[1/2])$. Using ideas from \cite{karoubi2015witt} we show that
this category satisfies descent in the ``real étale topology'' in an appropriate
sense. This allows us to assume that $k$ is real closed. Using ideas from
semi-algebraic topology, we prove that the category $\DM_W(k, \ZZ[1/2])$ is actually
independent of the real closed field $k$, namely that $\DM_W(k, \ZZ[1/2])$ is
canonically equivalent to $D(\ZZ[1/2])$. Putting all of
these facts together, we obtain the conservativity theorem for orderable fields.

The paper concludes with two appendices. In Appendix \ref{sec:compact-objects},
we recall some well-known results about compact objects in abelian and
triangulated categories. In Appendix \ref{sec:constructing-F*} we
construct the missing filtration $\tilde{F}_\bullet$ mentioned earlier. (The
results from appendix \ref{sec:compact-objects} are needed to establish that the
filtration is finite.) This
appendix is essentially a long and detailed computation with strictly homotopy invariant
sheaves. As a bonus, we can reprove a version of rigidity for torsion objects in
motivic homotopy theory \cite{rigidity-in-motivic-homotopy-theory}, see
Corollary \ref{corr:rigidity}.

The author wishes to thank Fabien Morel for suggesting this topic of
investigation and providing many helpful insights. He also wishes to thank
Mikhail Bondarko for useful discussions, and in particular for showing to him
Lemma \ref{lemm:bondarko}.

\paragraph{Notations and Conventions}
The signs $\wequi$ and $\iso$ shall denote isomorphisms. Whenever there is a category
$\mathcal{C}$ together with a homotopy category $Ho(\mathcal{C})$, we shall only
apply $\wequi$ to mean an isomorphism in $Ho(\mathcal{C})$ and $\iso$ to mean
an isomorphism in $\mathcal{C}$.

An adjunction between categories $\mathcal{C}, \mathcal{D}$ with left adjoint
$F$ and right adjoint $U$ shall be denoted $F: \mathcal{C} \leftrightarrows
\mathcal{D}: U$.

\section{The Abstract Hurewicz Theorem}
\label{sec:abstract-hurewicz}

In this short section we prove an abstract result about $t$-categories, by which
we mean triangulated categories provided with a $t$-structure
\cite[Section 1.3]{beilinson1982faisceaux}. It is
(we feel) so natural that it has probably occurred elsewhere before. We also
include some well-known results. Since all proofs are short and easy, we include
them for the convenience of the reader.

We first review some notations. The most important thing to point out is that we
use \emph{homological} notation for $t$-categories, whereas most sources, such
as the excellent \cite[Section IV \S 4]{gelfand-methods},
use cohomological notation. The usual reindexing
trick ($E_n \leftrightarrow E^{-n}$) can be used to translate between the two.

In more detail, a $t$-category is a triple $(\mathcal{C}, \mathcal{C}_{\le 0},
\mathcal{C}_{\ge 0})$ consisting of a triangulated category $\mathcal{C}$ and
two full subcategories $\mathcal{C}_{\le 0}, \mathcal{C}_{\ge 0} \subset
\mathcal{C}$ satisfying certain properties. One puts $\mathcal{C}_{\le n} =
\mathcal{C}_{\le 0}[n]$ and $\mathcal{C}_{\ge n} = \mathcal{C}_{\ge 0}[n]$. The
inclusion $\mathcal{C}_{\le n} \hookrightarrow \mathcal{C}$ has a
\emph{left} adjoint denoted $E \mapsto E_{\le n}$. Similarly the inclusion $\mathcal{C}_{\ge
n} \hookrightarrow \mathcal{C}$ has a \emph{right} adjoint denoted $E
\mapsto E_{\ge n}$. Both are called \emph{truncation}.

The full subcategory $\mathcal{C}^\heart = \mathcal{C}_{\ge 0} \cap
\mathcal{C}_{\le 0}$ turns out to be abelian. It is called the \emph{heart} of
the $t$-category $\mathcal{C}$. The \emph{homotopy objects} are
$\pi_0^\mathcal{C}(E) = (E_{\le 0})_{\ge 0} \wequi (E_{\ge 0})_{\le 0} \in
\mathcal{C}^\heart$ and $\pi_i^\mathcal{C}(E) = \pi_0^\mathcal{C}(E[-i])$.

One has many more properties. For example $\mathcal{C}_{\ge n} \supset
\mathcal{C}_{\ge n+1}$, $\mathcal{C}_{\le n} \subset \mathcal{C}_{\le n+1}$,
there are functorial distinguished triangles $E_{\ge n} \to E \to E_{\le n-1}$,
the functor $\pi_*^\mathcal{C}$ is homological (turns distinguished triangles
into long exact sequences), etc.

Recall that a triangulated functor $F: \mathcal{C} \to \mathcal{D}$ between $t$-categories
is called \emph{right-t-exact} (respectively \emph{left-t-exact}) if
$F(\mathcal{C}_{\ge 0}) \subset \mathcal{D}_{\ge 0}$ (respectively if
$F(\mathcal{C}_{\le 0}) \subset \mathcal{D}_{\le 0}$). It is called
\emph{t-exact} if it is both left and right $t$-exact.

If $F: \mathcal{C} \to \mathcal{D}$ is any triangulated functor between
$t$-categories, then it induces a functor $F^\heart: \mathcal{C}^\heart \to
\mathcal{D}^\heart, E \mapsto \pi_0^\mathcal{D} FE$.

For computations, we often write $[A, B]$ instead of $\Hom_\mathcal{C}(A, B)$.
We also note that the axioms for $t$-categories are self-dual. Passing from
$\mathcal{C}$ to $\mathcal{C}^{op}$ (and $\mathcal{D}$ to $\mathcal{D}^{op}$)
turns left-$t$-exact functors into right-$t$-exact functors, etc.

The notions of left and right $t$-exactness are partly justified by the
following.

\begin{prop} \label{prop:adjoints-heart}
Let $M: \mathcal{C} \rightleftarrows \mathcal{D}: U$ be adjoint functors between
$t$-categories, with $M$ right-$t$-exact and $U$ left-$t$-exact. Then
\[ M^\heart: \mathcal{C}^\heart \rightleftarrows \mathcal{D}^\heart: U \]
is also an adjoint pair. In particular $M^\heart$ is right exact and $U^\heart$
is left exact.
\end{prop}
\begin{proof}
Let $A \in \mathcal{C}^\heart, B \in \mathcal{D}^\heart$. We compute
\begin{align*}
[M^\heart A, B] &\stackrel{(1)}{=} [\pi_0^\mathcal{D} MA, B]
    \stackrel{(2)}{=} [(MA)_{\le 0}, B] \stackrel{(3)}{=} [MA, B] \\
  &\stackrel{(4)}{=} [A, UB]
   \stackrel{(3')}{=} [A, (UB)_{\ge 0}]
   \stackrel{(2')}{=} [A, \pi_0^\mathcal{C} UB] \stackrel{(1')}{=} [A, U^\heart B].
\end{align*}
Here (1) is by definition, (2) is because $A \in \mathcal{C}_{\ge 0}$ and so $MA
\in \mathcal{D}_{\ge 0}$ by right-$t$-exactness of $M$, and (3) because $B \in
\mathcal{D}_{\le 0}$. Equality (4) is by adjunction, and finally (3'), (2'),
(1') reverse (3), (2), (1) with $M \leftrightarrow U$, left $\leftrightarrow$
right, etc.
\end{proof}

The following lemma is a bit technical but naturally isolates a crucial step in
the proof of the main theorem of this section.

\begin{lemm} \label{lemm:technical-t-stuff}
Let $M: \mathcal{C} \rightleftarrows \mathcal{D}: U$ be adjoint functors between
$t$-categories. If $M$ is right-$t$-exact (respectively $U$ left-$t$-exact) then
there is a natural isomorphism $(UE)_{\ge n} \wequi (UE_{\ge n})_{\ge n}$
(respectively $(ME)_{\le n} \wequi (M E_{\le n})_{\le n}$).
\end{lemm}
\begin{proof}
By duality, we need only prove one of the statements. Suppose that $U$ is
left-$t$-exact. We compute for $T \in \mathcal{D}_{\le n}$
\begin{align*}
[(ME)_{\le n}, T] &\stackrel{(1)}{=} [ME, T] \stackrel{(2)}{=} [E, UT] \\
   &\stackrel{(3)}{=} [E_{\le n}, UT] \stackrel{(2)}{=} [M E_{\le n}, T]
    \stackrel{(1)}{=} [(M E_{\le n})_{\le n}, T].
\end{align*}
Since all equalities are natural, the result follows from the Yoneda lemma. Here
(1) is because $T \in \mathcal{D}_{\le n}$, (2) is by adjunction, and (3) is because
$U$ is left-$t$-exact so $UT \in \mathcal{C}_{\le n}$.
\end{proof}

With this preparation, we can formulate our theorem.

\begin{thm}\label{thm:abstract-hurewicz}
Let $M: \mathcal{C} \rightleftarrows \mathcal{D}: U$ be adjoint functors between
$t$-categories, such that $M$ is right-$t$-exact and $U$ is left-$t$-exact. Then
for $E \in \mathcal{C}_{\ge 0}$ there is a natural isomorphism
\[ \pi_0^\mathcal{D} ME \wequi M^\heart \pi_0^\mathcal{C} E. \]
\end{thm}
\begin{proof}
We have
\[ M^\heart \pi_0^\mathcal{C} E = \pi_0^\mathcal{D} M \pi_0^\mathcal{C} E
  \stackrel{(1)}{=} \pi_0^\mathcal{D} M E_{\le 0}
  \stackrel{(2)}{=} (M E_{\le 0})_{\le 0} \stackrel{(3)}{=} (ME)_{\le 0}
  \stackrel{(2)}{=} \pi_0^\mathcal{D} ME. \]
This is the desired result.
Here (1) is because $E \in \mathcal{C}_{\ge 0}$, (2) is because $M$ is
right-$t$-exact, and (3) is because of
Lemma \ref{lemm:technical-t-stuff} applied to left-$t$-exactness of $U.$
\end{proof}

This has the following useful corollary.

\begin{corr} \label{corr:hurewicz-conservativity}
In the above situation, if $M^\heart$ (or $U^\heart$) is an equivalence of
categories, and the $t$-structure on $\mathcal{C}$ is non-degenerate,
then the functor $M$ is conservative for connective objects.
\end{corr}
\begin{proof}
If $X \in \mathcal{C}_{\ge n}$ and $FX = 0$ then $F^\heart(\pi_n^\mathcal{C} X)
= \pi_n^\mathcal{D}FX = 0$, so $X \in \mathcal{C}_{\ge n+1}$. Iterating this
argument we find that $\pi_i^\mathcal{C}X = 0$ for all $i$ and so $X = 0$ by
non-degeneracy of the $t$-structure.
\end{proof}

For convenience, we also include the following well-known observation.

\begin{lemm} \label{lemm:t-exactness-adjoints}
Let $M: \mathcal{C} \rightleftarrows \mathcal{D}: U$ be adjoint functors between
$t$-categories. Then $M$ is right-$t$-exact if and only if $U$ is
left-$t$-exact.
\end{lemm}
So in applying the theorem, only one of the two exactness properties has to be
checked.
\begin{proof}
This is immediate by adjunction: $M$ is right-$t$-exact if and only if
$M(\mathcal{C}_{\ge 0}) \subset \mathcal{D}_{\ge 0}$, which happens
if and only if $[ME, F] = 0$ for all $E \in \mathcal{C}_{\ge 0}, F \in \mathcal{D}_{< 0}$.
By adjunction
$[ME, F] = [E, UF]$, and this vanishes if and only if $U(\mathcal{D}_{< 0}) \subset
\mathcal{C}_{< 0}$, i.e. $F$ is left-$t$-exact.
\end{proof}

To illustrate these rather abstract results, we show how to recover the
classical Hurewicz theorem for spectra. For this, we consider understood the
functor $C_*: \SH \to D(Ab)$ and the respective $t$-structures. We write $\pi_i
= \pi_i^\SH, H_i = \pi_i^{D(Ab)}$. Since $C_*$ commutes with arbitrary sums
(wedges), it has a right adjoint $U$, by Neeman's version of Brown
representability. Since $C_* S = \ZZ[0]$ is projective, an easy computation
shows that $\pi_i UC = H_i C$ for any $C \in D(Ab)$. Thus $U$ is $t$-exact (it is
in fact the Eilenberg-MacLane spectrum functor). From Lemma
\ref{lemm:t-exactness-adjoints} and Theorem \ref{thm:abstract-hurewicz} we
conclude now that for $E \in \SH_{\ge 0}$ we have $\pi_0 E = H_0 E$ (recall that
$C_*^\heart$ is an isomorphism).

\paragraph{Constructing $t$-structures.}
In Section \ref{sec:witt-motives} we will have the need of transferring
$t$-structures to localisations. The following Lemma applies in many situations
where the homotopy objects are computed as homotopy sheaves, as will always be
the case for us.

\begin{lemm} \label{lemm:transfer-$t$-structure}
Let $\mathcal{C}$ be a compactly generated triangulated $t$-category with coproducts,
$G \subset \mathcal{C}$ a set of compact generators.

Let $i: U \subset \mathcal{C}$ be a colocalising subcategory with left adjoint $L: \mathcal{C} \to U$.

Assume the following:
\begin{enumerate}[(a)]
\item $G \subset \mathcal{C}_{\ge 0}$,
\item if $E \to F$ is a morphism in $\mathcal{C}$ such that for all $X \in G$
  the induced map $[X, E] \to [X, F]$ is surjective, then $\pi_0^\mathcal{C}(E) \to
  \pi_0^\mathcal{C}(F)$ is also surjective,
\item the $t$-structure is non-degenerate,
\item homotopy objects commute with directed homotopy colimits,
\item $L$ preserves compact objects,
\item $iLG \subset \mathcal{C}_{\ge 0}$.
\end{enumerate}

Then $U$ has a unique $t$-structure such that $i: U \to \mathcal{C}$ is $t$-exact.
\end{lemm}
\begin{proof}
Uniqueness is clear. We need to show existence.

I claim that for $E \in U$ there exists a map $E \to E'$ with $E' \in U$,
$\pi_0^\mathcal{C}(E') = 0$ and $\pi_i^\mathcal{C}(E) = \pi_i^\mathcal{C}(E')$
for $i < 0$. Indeed let $A = \bigoplus_{X \in G} \bigoplus_{[X, E]} X$. Then
there is a canonical map $A \to E = iE$ and the cofibre $LA \to E \to E'$ has
the required property. Indeed $\pi_i^\mathcal{C}(E') = \pi_i^\mathcal{C}(E)$ for
$i < 0$ because $\pi_i^\mathcal{C}(A) = 0$ by assumptions (a), (d). Also for $X \in G$,
given $\alpha: X \to iE$ we get a factorisation $X \to iLX \to iE$. Thus $[X,
iLA] \to [X, iE]$ is surjective and by assumption (b), $\pi_0^\mathcal{C}(E') = 0$.

Repeating the argument, given $E \in U$ we get a diagram $E = E_0 \to E_1 \to E_2 \to
\dots$ with $E_i \in U$, and
where $E_i \to E_{i+1}$ has the property that $\pi_k^\mathcal{C}(E_i) =
\pi_k^\mathcal{C}(E_{i+1})$ for $k < i$ and $\pi_i^\mathcal{C}(E_{i+1}) = 0$.
Let $E_\infty = \hocolim_i E_i$. Since $L$ preserves compact objects by (e), $i$
preserves coproducts and so commutes with hocolim, whence $E_\infty$ makes
unambiguous sense (i.e. we may compute the homotopy colimit in $U$ or
$\mathcal{C}$, with the same result).

But note that $\pi_i^\mathcal{C}(E) =
\pi_i^\mathcal{C}(E_k)$ for all $i < 0$ and all $k \ge 0$, whereas
$\pi_i^\mathcal{C}(E_k) = 0$ for $0 \le i < k$.
It follows from our assumption (d) that $\pi_i^\mathcal{C}(E_\infty) = 0$ for all $i \ge 0$, and
thus by non-degeneracy of the $t$-structure (c) we have $E_\infty \in \mathcal{C}_{<
0}$. The map $E \to E_\infty$ induces an isomorphism on $\pi_i^\mathcal{C}$ for
all $i < 0$ and hence $E_\infty \wequi E_{< 0}$, again by (c).

Consequently we have shown that for $E \in U$ we also have $E_{< 0} \in U$. It
follows that $E_{\ge  0} \in U$ (since $U$ is triangulated) and one easily
verifies that $U \cap
\mathcal{C}_{\ge 0}, U \cap \mathcal{C}_{< 0}$ defines a $t$-structure on $U$.
The functor $i: U \to \mathcal{C}$ is $t$-exact by design.
\end{proof}

\section{The Case of Motivic Homotopy Theory:
  The Motivic Hurewicz Theorem}
\label{sec:motivic-homotopy-theory}

In this section we shall work with a fixed perfect ground field $k$.

We now show how to apply the results of the previous section to motivic homotopy
theory. This mainly consists in recalling definitions and providing references.

First we need to recall the construction of $\SH(k)$ and $\DM(k)$. We follow \cite[Section
2]{RondigsModules}. Let $Sm(k)$ be the category of smooth schemes over the
perfect field $k$ and $Cor(k)$ the category whose objects are the smooth schemes
and whose morphisms are the finite correspondences. We write $Shv(k)$
(respectively $Shv^{tr}(k)$) for the categories of Nisnevich sheaves. Write $R:
Sm(k) \to Shv(k)$ and $R_{tr}: Cor(k) \to Shv^{tr}(k)$ for the functors sending an
object to the sheaf it represents.

There is a natural functor $M: Sm(k) \to Cor(k)$ with $M(X) = X$ and $M(f) =
\Gamma_f$, the graph of $f$. This induces a functor $U: Shv^{tr}(k) \to Shv(k)$
via $(UF)(X) = F(MX)$. There is a left adjoint $M: Shv(k) \to Shv^{tr}(k)$ to
$U$. It is the unique colimit-preserving functor such that $M(RX) = R_{tr}(X)$.
Write $Shv_*(k)$ for the category of pointed sheaves. Then there is $R_+: Sm(k)
\to Shv_*(k)$ obtained by adding a disjoint base point. The objects in
$U(Shv^{tr}(k))$ are canonically pointed (by zero) and one obtains a new
adjunction $\tilde{M}: Shv_*(k) \rightleftarrows Shv^{tr}(k): U$

We can pass to simplicial objects and extend $M$ and $U$ levelwise to obtain an
adjunction $\tilde{M}: \Delta^{op} Shv_*(k) \rightleftarrows \Delta^{op} Shv^{tr}(k) : U$.
We denote by $Spt(k)$ the category of $S^{2,1} := S^1 \wedge \Gm$-spectra in
$\Delta^{op} Shv_*(k)$ and by $Spt^{tr}(k)$ the category of $M(S^{2,1})$-spectra
in $\Delta^{op} Shv^{tr}(k)$. The adjunction still extends, so we obtain the
following commutative diagram.

\begin{figure*}[h]
\centering
\begin{tikzcd}
  Sm(k) \arrow{d}{R_+} \arrow{r}{M} & Cor(k) \arrow{d}{R_{tr}} \\
  Shv_*(k) \arrow[bend left=10]{r}{\tilde{M}} \arrow{d}
            & \arrow[bend left=10]{l}{U} Shv^{tr}(k) \arrow{d}  \\
  \Delta^{op} Shv_*(k) \arrow[bend left=10]{r}{\tilde{M}} \arrow{d}{\Sigma^\infty}
           & \arrow[bend left=10]{l}{U} \Delta^{op} Shv^{tr}(k) \arrow{d}{\Sigma^\infty} \\
  Spt(k) \arrow[bend left=10]{r}{M} & \arrow[bend left=10]{l}{U} Spt^{tr}(k)
\end{tikzcd}
\end{figure*}

(See also \cite[Diagram (4.1)]{hoyois-algebraic-cobordism}.)
One may put the projective local model structures on the four categories in the
lower square and then the adjunctions become Quillen adjunctions, so pass
through localisation. Contracting the affine line yields the $\Aone$-local model
structures. The homotopy category of $Spt(k)$ (in this model structure) is
denoted $\SH(k)$ and is called the motivic stable homotopy category. Similarly the
homotopy category of $Spt^{tr}(k)$ is denoted $\DM(k)$. It is essentially a
bigger version of the category constructed by Voevodsky, as explained in
\cite[Section 2]{RondigsModules}. We have thus the following commutative
diagram.

\begin{figure*}[h]
\centering
\begin{tikzcd}
  Sm(k) \arrow{d}{\Sigma^\infty(\bullet_+)} \arrow{r}{M} & Cor(k) \arrow{d} \\
  \SH(k) \arrow[bend left=10]{r}{M} & \arrow[bend left=10]{l}{U} \DM(k)
\end{tikzcd}
\end{figure*}

We recall that $M$ (in all its incarnations) is a symmetric monoidal functor.

Next we need to define $t$-structures on $\DM(k)$ and $\SH(k)$. For $E \in
\SH(k)$ (respectively $E \in \DM(k)$) let $\ul \pi_i(E)_j$ (respectively $\ul
h_i(E)_j$) be the sheaf on $Sm(k)$ (respectively on $Cor(k)$) associated with the
presheaf $V \mapsto [\Sigma^\infty (V_+) \wedge S^i, \Gm^{\wedge j} \wedge E]$.
Put
\begin{align*}
  \SH(k)_{\ge 0} &= \{E \in \SH(k): \ul \pi_i(E)_j = 0 \text{ for $i < 0$ and $j \in \ZZ$}\} \\
  \SH(k)_{\le 0} &= \{E \in \SH(k): \ul \pi_i(E)_j = 0 \text{ for $i > 0$ and $j \in \ZZ$}\},
\end{align*}
and similarly for $\DM(k)$. By \cite[Section 5.2]{morel-trieste} this defines a
$t$-structure on $\SH(k)$ called the \emph{homotopy $t$-structure}. It is also
true that $\DM(k)_{\le 0}, \DM(k)_{\ge 0}$ define a $t$-structure. This can be
seen by repeating the arguments of
\cite[Section 2.1]{hoyois-algebraic-cobordism} for $\DM(k)$, noting that the connectivity
theorem for $\DM(k)$ follows from Voevodsky's cancellation theorem.

With all this setup out of the way, we can prove the result of this section.

\begin{lemm} The functor $U: \DM(k) \to \SH(k)$ is $t$-exact.
In fact for $E \in \DM(k)$ we have $\ul \pi_i(UE)_j = U(\ul h_i(E)_j)$ and $U:
Shv^{tr}(k) \to Shv(k)$ detects zero objects.
\end{lemm}
\begin{proof}
It follows from the definitions of the $t$-structures that we need only prove
the ``in fact'' part. Let $\ul \pi_i^{pre}(E)_j$ be the presheaf
$V \mapsto [\Sigma^\infty (V_+) \wedge S^i, \Gm^j \wedge E]$, and similarly for
$\ul h_i^{pre}(E)_j$. Then writing also $U: PreShv(Cor(k)) \to PreShv(Sm(k))$
we get immediately from the definitions that $U(\ul h_i^{pre}(E)_j) = \ul
\pi_i^{pre}(UE)_j$. So we need to show that $U: PreShv(Cor(k)) \to
PreShv(Sm(k))$ commutes with taking the associated sheaf. This is well known,
see e.g. \cite[Theorem 13.1]{lecture-notes-mot-cohom}.
\end{proof}

\begin{corr}[Preliminary form of the Motivic Hurewicz Theorem]
\label{corr:preliminary-hurewicz}
Let $E \in \SH(k)_{\ge 0}$. Then $ME \in \DM(k)_{\ge 0}$ and $\ul h_0(ME)_* =
M^\heart(\ul \pi_0(E)_*)$.
\end{corr}
Here $\ul \pi_0(E)_*$ denotes the homotopy object in $\SH(k)^\heart$, and
similarly for $\ul h_0(ME)_*$.
\begin{proof}
We know that $M$ is left adjoint to the $t$-exact functor $U$. Hence $M$ is
right-$t$-exact by Lemma \ref{lemm:t-exactness-adjoints}. Thus $ME \in
\DM(k)_{\ge 0}$, and the result about homotopy objects is
just a concrete incarnation of Theorem \ref{thm:abstract-hurewicz}.
\end{proof}

Next we explain how to identify $M^\heart: \SH(k)^\heart \to
\DM(k)^\heart$. To do this, recall the Hopf map $\mathbb{A}^2 \setminus\{0\} \to
\mathbb{P}^1$. In $\SH(k)$ we have the isomorphism $\Sigma^\infty(\mathbb{A}^2
\setminus \{0\}) \wequi \Sigma^\infty(\mathbb{P}^1 \wedge \Gm)$ and hence this defines a stable
map $\eta: \Sigma^\infty(\Gm) \to S$, where $S$ is the sphere spectrum
$\Sigma^\infty(Spec(k)_+)$. If $F \in \SH(k)^\heart$ we put $F\gmtw{n} := \ul \pi_0(F
\wedge \Gm^{\wedge n})$. Consequently there is a natural map $\eta_F = \ul
\pi_0(\eta \wedge \id_F): F\gmtw{1} \to F$.

An important observation is that $M(\eta)$ is the zero map. One may show that
this implies that for
$F \in \DM(k)^\heart$, we have $0 = \eta_{UF}: UF\gmtw{1} \to UF$. We denote by
$\SH(k)^{\heart,\eta=0}$ the full subcategory of $\SH(k)^\heart$ consisting of
objects $F \in \SH(k)^\heart$ with $\eta_F = 0$.

\begin{thm}[Deglise \cite{deglise-htpy-modules}] \label{thm:deglise}
Let $k$ be a perfect field.
The functor $U: \DM(k)^\heart \to \SH(k)^{\heart, \eta=0}$ is an equivalence of
categories.
\end{thm}
\begin{proof}
Modulo identifying $\SH(k)^\heart = \Pi_*(k)$ and $\DM(k)^\heart =
\Pi^{tr}_*(k)$ this is Theorem 1.3.4 of Deglise. The first identification is
explained in \cite[Theorem 5.2.6]{morel-trieste}. The second one is obtained by
adapting loc. cit.
\end{proof}

\begin{corr}
For $F \in \SH(k)^{\heart,\eta=0}$ we have $UM^\heart F = F$.
\end{corr}
\begin{proof}
By the theorem we may write $F = UF'$. Using the fact that $M^\heart$ is left
adjoint to $U$ ($= U^\heart$) by Proposition \ref{prop:adjoints-heart},
we compute $[M^\heart F, T] = [M^\heart UF', T] =
[UF', UT] = [F', T]$, where the last equality is because $U$ is fully faithful (by the
theorem). Thus $M^\heart F = F'$ by the Yoneda lemma, and finally $UM^\heart F = UF' = F$.
\end{proof}

\begin{corr} \label{corr:Mheart}
For $F \in \SH(k)^\heart$ we have $UM^\heart(F) = F/\eta$, where $F/\eta$
denotes the cokernel of $\eta_F: F\gmtw{1} \to F$ in the abelian category
$\SH(k)^\heart$.
\end{corr}
\begin{proof}
We have the right exact sequence
\[ F\gmtw{1} \xrightarrow{\eta} F \to F/\eta \to 0. \]
Since $M^\heart$ is left adjoint it is right exact. Also $U$ is exact, so we get the
right exact sequence
\[ UM^\heart F\gmtw{1} \to UM^\heart F \to UM^\heart (F/\eta) \to 0. \]
The first arrow is zero and $UM^\heart(F/\eta) = F/\eta$ by the previous corollary
(note that $F/\eta \in \SH(k)^{\heart,\eta=0}$). The result follows.
\end{proof}

We thus obtain the Hurewicz theorem for $\SH(k) \to \DM(k)$.
\begin{thm}[Final Version of the Motivic Hurewicz Theorem] \label{thm:hurewicz}
Let $k$ be a perfect field and $E \in \SH(k)_{\ge 0}$.

Then $ME \in \DM(k)_{\ge 0}$ and modulo the identification of $\DM(k)^\heart$ as a full subcategory of
$\SH(k)^\heart$ (via Theorem \ref{thm:deglise}) we have
\[ \ul h_0(ME)_* = \ul \pi_0(E)_*/\eta. \]
\end{thm}
\begin{proof}
Combine Corollary \ref{corr:preliminary-hurewicz} with Corollary \ref{corr:Mheart}.
\end{proof}

\section{Homotopy Modules and the Slice Filtration}
\label{sec:homotopy-modules}

In essentially all of this section, the base field $k$ will be assumed perfect.
We restate this assumption with each theorem, but not necessarily otherwise.

We now study in more detail the heart $\SH(k)^\heart$.
First recall some notation. By $Sm(k)$ we denote the symmetric monoidal category
of smooth varieties over a field $k$, monoidal operation being cartesian
product. We write $Shv(k)$ for the abelian closed symmetric monoidal category of
sheaves of abelian groups on $Sm(k)$ in the Nisnevich topology. The monoidal
product comes from the ordinary tensor product of abelian groups. There is a
natural functor $\ZZ\bullet: Sm(k) \to Shv(k)$ sending a smooth variety $Y$ to the
sheaf associated with the presheaf $X \mapsto \ZZ\Hom(X,Y)$. Here $\ZZ\Hom(X,
Y)$ denotes the free abelian group on the set $\Hom(X,Y)$. (In fact this
presheaf is already a Nisnevich sheaf, but we do not need this observation.)
One has $\ZZ X \otimes \ZZ Y \iso \ZZ(X \times Y)$, i.e.
$\ZZ\bullet$ is a monoidal functor.

If $X \in Sm(k)$ and $x \in X$ is a rational point, there is a natural splitting
$\ZZ X \wequi \ZZ(Spec(k)) \oplus \ZZ(X, x)$ (defining the last term). We put
$\ZZ \Gm = \ZZ(\Aone \setminus \{0\}, 1)$.
An easy computation shows that for any $F \in Shv(k), U \in Sm(k)$ one has a
natural splitting $F(U \times \Gm) = F(U) \oplus \Omega F(U)$ (defining the last
term). In fact $\Omega F$ is easily seen to be a sheaf. One verifies without
difficulty that $\Omega F \wequi \iHom(\ZZ \Gm, F)$. (Here $\iHom$ denotes the
right adjoint of $\otimes$.) In other works the notation $\Omega F = F_{-1}$ is
often used; we avoid this for sake of clarity.

A sheaf $F \in Shv(k)$ is called \emph{strictly $\Aone$-invariant} (or
\emph{strictly invariant} if the context is clear) if for all $U \in Sm(k), n
\in \NN$ the natural map $H^n_{Nis}(U, F) \xrightarrow{p^*} H^n_{Nis}(U \times
\Aone, F)$ is an isomorphism. The full subcategory of $Shv(k)$ consisting of
strictly invariant sheaves is denoted $\HI(k)$.

By a \emph{homotopy module} we mean a collection $F_* \in \HI(k), * \in \ZZ$
together with isomorphisms $F_n \xrightarrow{\wequi} \Omega F_{n+1}$. A
morphism of homotopy modules $f_*: F_* \to G_*$ is a collection of morphisms
$f_n: F_n \to G_n$ such that $\Omega f_{n+1} = f_n$ under the natural
identifications. We denote the category of homotopy modules by $\HI_*(k)$. We
study it because of the following result.

\begin{thm}[Morel \cite{morel-trieste}, Theorem 5.2.6] \label{thm:identifying-SHheart}
Let $k$ be perfect and $E \in \SH(k)$. Then for each $n \in \ZZ$, the collection $\ul \pi_n(E)_*$
is a homotopy module. The functor $\SH(k)^\heart \to \HI_*(k), E \mapsto \ul \pi_0(E)_*$
is an equivalence of categories.
\end{thm}

In particular the category $\HI_*(k)$ is abelian. We can clarify the abelian
structure as follows. Recall that $\HI(k)$ is the heart of the $S^1$-stable
$\Aone$-homotopy category \cite[Lemma 4.3.7(2)]{morel-trieste}. By construction,
the inclusion $\HI(k) \to Shv(k)$ is exact. This means that for any morphism
of strictly invariant sheaves, the kernel and image (computed in $Shv(k)$) are
strictly invariant. Finally the functor $\omega^n: \HI_*(k) \to \HI(k), F_*
\mapsto F_n$ is also exact.

Homotopy modules have a lot of structure. The zeroth homotopy module of the
sphere spectrum is denoted $\ul K_*^{MW} := \ul \pi_0(S)_*$ and called
\emph{unramified Milnor-Witt K-theory}. It has been explicitly described by Morel
\cite[Chapter 3]{A1-alg-top}. For $n \ge 1$ there exist natural surjections $\ZZ \Gm^{\otimes n}
\twoheadrightarrow \ul K_n^{MW}$. Let $F_* \in \HI_*(k), m \in \ZZ$. The
isomorphism $\Omega F_m \wequi F_{m-1}$ induces by adjunction a pairing $\ZZ
\Gm \otimes F_{m - 1} \to F_m$. It turns out that this pairing factorises
through the surjection $\ZZ \Gm \twoheadrightarrow \ul K_1^{MW}$.

Now quite generally given any three sheaves $F, G, H$ there exists a natural
morphism
$\iHom(F, G) \otimes H \to \iHom(F, G \otimes H)$. Thus the pairing $\ul K_1^{MW}
\otimes F_m \to F_{m+1}$ induces $\Omega^2 \ul K_1^{MW} \otimes F_m \to \Omega^2
F_{m+1}$, i.e. $\ul K_{-1}^{MW} \otimes F_m \to F_{m-1}$. The graded ring sheaf
$\ul K_*^{MW}$ is generated by $\ul K_1^{MW}$ and $\eta \in \ul K_{-1}^{MW}$, so
there is at most one possible extension of these two pairings to a total pairing
$\ul K_*^{MW} \otimes F_* \to F_{* + *}$. It turns out that this paring always
exists. In fact $\HI_*(k) \wequi \SH(k)^\heart$ inherits a monoidal structure we
denote $\wedge$ to avoid confusion. For any $F, G \in \HI_*(k)$ there exists a
natural morphism $F_n \otimes G_m \to (F \wedge G)_{n+m}$. Since $\ul K_*^{MW} =
\ul \pi_0(S)_*$ is the unit object, we obtain a pairing $\ul K_n^{MW} \otimes F_m
\to (\ul K^{MW} \wedge F)_{n+m} \wequi F_{n+m}$. One may prove that these two
pairings just constructed coincide.

Homotopy modules also have transfers, though in a weaker sense than Voevodsky's
sheaves with transfers. For this, denote by $\mathcal{F}_k$ the
category of fields of finite transcendence degree over $k$. If $L \in
\mathcal{F}_k$ (or more generally $L$ essentially smooth over $k$) and $F \in
Shv(k)$, then $F(L)$ can be defined by an appropriate colimit. If $L'/L$ is any
finite extension (still $L \in \mathcal{F}_k$), $n \in \ZZ$ and $F_* \in \HI_*(k),$ then
Morel has constructed the \emph{cohomological transfer} \cite[Chapter 4]{A1-alg-top}
$tr_{L'/L}: F_n(L') \to F_n(L)$. (Actually, this only works in characteristic
not two. See Subsection \ref{subsec:interaction-boundary-transfer} for the
correct definition in full generality.)

If $F$ is a strictly invariant sheaf then $F$ is unramified
\cite[Lemma 6.4.4]{morel2005stable}. (In particular
for connected $X \in Sm(k),$ the natural map $F(X) \to F(k(X))$ is injective.The
notion of unramified presheaves is reviewed in more detail at the beginning of
Subsection \ref{subsec:app-overview}.)
Thus $F = 0$ if and only if $F(L) = 0$ for all $L \in \mathcal{F}_k$. Since
$\HI(k)$ is abelian it follows that if $F$ is a strictly invariant sheaf and $G,
H \subset F$ are strictly invariant subsheaves, then $G = H$ if and only if
$G(L) = H(L)$ for all $L \in \mathcal{F}_k$ (see the beginning of Subsection
\ref{subsec:app-overview} for the definition of $F(L)$ for $L$ not of finite
type over $k$).

Now let $F \otimes G \to H$ be a pairing, where $H$ is a strictly invariant
sheaf with transfers for finite field extensions. We shall write
\[ (FG)^{tr}(L) := \langle tr_{L'/L} (F(L') G(L')) \rangle_{L'/L \text{ finite}} \subset H(L). \]
It follows from the above discussion that there exists at most one
strictly invariant sheaf $(FG)^{tr} \subset H$ with the above sections over fields.

Now recall the slice filtration
\cite[Section 2]{voevodsky-slice-filtration}. Write
$\SH(k)^{eff}(i)$ for the localising subcategory of $\SH(k)$ generated by
$(\Sigma^\infty X_+)\wedge \Gm^{\wedge i}$ for all $X \in Sm(k)$. The inclusion
$\SH(k)^{eff}(i) \hookrightarrow \SH(k)$ commutes with arbitrary sums by
construction and so has a right adjoint $f_i$ by Neeman's version of Brown
representability. The object $f_i E$ is called the $i$-th \emph{slice cover} of
$E$. It is easy to see that there is a commutative diagram of natural transformations
\begin{equation*}
\begin{CD}
f_i @>>> f_{i-1} \\
@VVV     @VVV    \\
\id @=   \id.
\end{CD}
\end{equation*}

We call $E$ such that $E \in \SH(k)^{eff}(n)$ for some $n$ (equivalently $E =
f_nE$) \emph{slice-connective}.

Suppose that $E \in \SH(k)_{\ge 0}$. We want to define a filtration on $\ul
\pi_0(E)_*$. We shall put
\[ F_N \ul \pi_0(E)_* := im(\ul \pi_0(f_{-N}E)_* \to \ul
     \pi_0(E)_*) \subset \ul \pi_0(E)_*. \]
(Image computed in the abelian category $\HI_*(k)$.)
There is now the following interesting result. (We warn the reader that Levine uses
somewhat different indexing conventions than we do.)
\begin{thm}[Levine \cite{levine-slice}, slightly adapted Theorem 2] \label{thm:levine-slice}
Let $k$ be a perfect field of characteristic different from $2$ and $E \in \SH(k)_{\ge 0}$. Then for
$m \ge i$ and any perfect field extension $F/k$ we have
\begin{equation} \label{eq:tr-closure-formula}
 (F_i \ul \pi_0(E)_*)_m(F) = (\ul K_{m-i}^{MW} \ul \pi_0(E)_i)^{tr}(F).
\end{equation}
\end{thm}

Note that if $k$ has characteristic zero, then the above theorem implies that
there is a (unique) strictly invariant sheaf $(\ul K_{m-i}^{MW} \ul
\pi_0(E)_i)^{tr}$ with sections over fields given by formula
\eqref{eq:tr-closure-formula}. In
characteristic $p>0$ the result is not quite strong enough, but it is good
enough if we invert $p$. In Appendix \ref{sec:constructing-F*}, we prove the
following variant.

\begin{thm}\label{thm:slice-filtration}
Let $H_* \in \HI_*(k)$ be a homotopy module over the perfect field $k$.
\begin{enumerate}[(1)]
\item For all $m \in \ZZ, n \ge 0$ there exists a (unique) strictly invariant
      sheaf
      $(\ul K_n^{MW} H_m)^{tr} \subset H_{n+m}$ which has sections
      over fields given by formula \eqref{eq:tr-closure-formula}.

\item We have that $(\ul K_n^{MW} H_m)^{tr} = \Omega(\ul K_{n+1}^{MW}
      H_m)^{tr} \subset H_{n+m}$ for all $m \in \ZZ, n \ge 0$.
\end{enumerate}

Put \[ (\tilde{F}_nH)_m = \begin{cases}
   (\ul K^{MW}_{m-n} H_n)^{tr} &: m > n \\
   H_m &: m \le n
\end{cases}. \]

\begin{enumerate}[(1)]
\setcounter{enumi}{2}
\item $(\tilde{F}_n H)_*$ is a homotopy submodule of $H_*$.
\item If $H = \ul\pi_0(E)_*$, for $E \in \SH(k)^{cpt}_{\ge 0}$, then the
      filtration $\tilde{F}_\bullet H$ is finite: there exists $N >> 0$ with $\tilde{F}_N H =
      H$.
\end{enumerate}
\end{thm}

(Here $\SH(k)^{cpt}_{\ge 0}$ denotes the set of compact objects in $\SH(k)_{\ge
0}$.)
What this theorem says is that even though we may not be able to understand
the filtration $F_\bullet H$, we can construct a different filtration
$\tilde{F}_\bullet H$ which agrees on perfect fields. There is one set-back
compared to Levine's theorem: the filtration $\tilde{F}_\bullet H$ is completely
algebraic, not geometric in origin. In particular if $E \in \SH(k)^{eff}(-N)_{\ge 0}$ for some
$N$ (i.e. $E$ is slice-connective),
then it is immediate that $F_N \ul \pi_0E_* = \ul \pi_0E_*$, but we do not have such
a nice characterisation for $\tilde{F}_\bullet$. However see (4) of the Theorem.

We write $\SH(k)_e$ for
the e-localised version of $\SH(k)$. This is the full (colocalising) monoidal 
triangulated subcategory of all objects $E \in \SH(k)$ such $E \xrightarrow{e}
E$ is an isomorphism; equivalently all the homotopy sheaves $\ul \pi_i(E)$ are
uniquely e-divisible.

We can finally prove our first conservativity theorem.

\begin{thm}[Conservativity I] \label{thm:conservativity-I}
Let $k$ be a perfect field of finite $2$-étale cohomological dimension and
exponential characteristic $e$, and $E
\in \SH(k)$. Assume that either (a) $E \in \SH(k)_e$ is
connective and slice-connective, or (b) $E$ is compact.

Then if $ME = 0$, one also has $E = 0$.

More specifically, if either (a) or (b) holds, and $ME \in \DM(k)_{\ge n}$ then
also $E \in \SH(k)_{\ge n}$.
\end{thm}
\begin{proof}
As a preparatory remark, let us note that compact objects are connective and
sllice-connective. Indeed $\SH(k)$ is generated (as a localising
subcategory) by the compact objects $\Sigma^\infty X_+(i)$ for $i \in \ZZ$ and
$X \in Sm(k)$ (see e.g. \cite{RondigsModules}, Lemma 2.27 and paragraph
thereafter); it follows from general results
\cite[Lemma 2.2]{Neeman1992} that $\SH(k)^{cpt}$ is the thick triangulated subcategory generated by
the same objects. It follows that $E \in \SH(k)^{cpt}$ is obtained using
finitely many operations from (finitely many) objects of the form $\Sigma^\infty
X_+(i)$, and all of these are connective \cite{morel2005stable} and
slice-connective. It remains to observe that the subcategories of connective and
slice-connective spectra are also thick.

We shall prove: if $E \in \SH(k)_{\ge n}$, (a) or (b) holds, and $\ul h_n(ME)_* = 0$ then $E
\in \SH(k)_{\ge n+1}$. This is the ``more specifically'' part. It follows that
if actually $ME = 0$ then $\ul \pi_i(E)_j = 0$ for all $i$ and $j$
and so $E \wequi 0$, since weak equivalences in $\SH(k)$ are detected by $\ul
\pi_*(\bullet)_*$, essentially by construction \cite[Proposition
5.1.14]{morel-trieste}.

The first claim of the theorem then follows, since under (a) we assume $E$
connective, and under (b) we assume $E$ compact and all compact objects are
connective. 

To explain the general argument, let us assume for now that $k$ has
characteristic zero. We will indicate at the end what needs to be changed in
positive characteristic.

By unramifiedness of strictly invariant sheaves, it is enough to show that $\ul
\pi_n(E)_*(K) = 0$ for all finitely generated field extensions $K/k$. We consider the
filtration $F_\bullet \ul \pi_n(E)_*$. We know that $F_N \ul \pi_n(E)_* = \ul
\pi_n(E)_*$ for some $N$ sufficiently large, by slice-connectivity. I claim that for $r
>> 0$ we have $\ul \pi_n(E)_{N + r}(K) = 0$. This claim proves the result.
Indeed we have $\ul \pi_n(E)_*/\eta = 0$ by the Hurewicz Theorem \ref{thm:hurewicz}, so $\eta$ is
surjective, whence the claim implies that $\ul \pi_n(E)_i(K) = 0$ for all $i$.

We compute
\begin{align*}
 \ul \pi_n(E)_{N+r}(K) &= (\ul K_r^{MW} \ul \pi_n(E)_N)^{tr}(K) \\
                &= (\ul K_r^{MW} \eta^r \ul \pi_n(E)_{N+r})^{tr}(K) \\
                &= (\ul I^r \ul \pi_n(E)_{N+r})^{tr}(K).
\end{align*}
Here the first equality is by Theorem \ref{thm:levine-slice},
the second is by surjectivity of
$\eta$, and the third is essentially by definition. Here we make use of the
\emph{sheaf of unramified fundamental ideals} $\ul{I}$ and its powers $\ul{I}^r$.
We can define it quickly as $\ul{I} = im(\eta: \ul{K}_1^{MW} \to
\ul{K}_0^{MW})$. It is easy to see that $\ul{I}^r = im(\eta^r:\ul{K}_r^{MW} \to
\ul{K}_0^{MW})$, and also that if $K$ is any field, then $\ul{I}^r(K) \subset
\ul{K}_0^{MW}(K) \iso GW(K)$ is given by $I(K)^r$, the $r$-th power of the
fundamental ideal. It is thus enough to show
that for some $r >> 0$ and for all $L/K$ finite, we have $I(L)^r = 0$.
Now $K$ is finitely generated over $k$, so $K$ has finite 2-étale cohomological dimension
\cite[Theorem 28 of Chapter 4]{shatz1972profinite}, say $R$.
Then $L$ also has 2-étale cohomological dimension bounded by $R$ by
loc. cit. Finally the resolution of the Milnor
conjectures (see \cite{morel-voevodskys-proof} for an overview)
implies that $I(L)^l = 0$ for $l > R$. This concludes the proof in
characteristic zero.

If $e = 2$ then the functor $\SH(k)_2 \to
\DM(k, \ZZ[1/2])$ is conservative on connective objects by Corollary
\ref{corr:hurewicz-conservativity} (see the discussion of $\SH(k)_2^-$
in the last paragraph of this section), and the theorem follows.

If $e > 1$ and assumption (b) holds (i.e. $E$ is compact), then the same
argument as in characteristic zero applies, but with $F_\bullet$ replaced by
$\tilde{F}_\bullet$ and the reference to Theorem \ref{thm:levine-slice} replaced
by Theorem \ref{thm:slice-filtration}.

If $e > 2$ and assumption (a) holds, we still need to show that for $K/k$
finitely generated we have $\ul \pi_n(E)_*(K)= 0$. Let $K^p/K$ be the perfect
closure. By Lemma \ref{lemm:inj-p} below, the natural map $\ul \pi_n(E)_*(K) \to
\ul \pi_n(E)_*(K^p)$ is injective. It is thus enough to show that the latter group is
zero. For this we use the same argument as in characteristic zero; the only
additional thing we need is to show that $K^p$ has finite 2-étale cohomological
dimension, but this follows from finite dimension of $K$ since they have the
same étale site \cite[Tags 04DZ and
01S4]{stacks-project}.
\end{proof}

\begin{lemm} \label{lemm:inj-p}
Let $k$ be a perfect field of characteristic $p > 0$
and $H_*$ a homotopy module on which $p$ is invertible. Let $K/k$ be a finitely
generated field and $K'/K$ a purely inseparable extension. Then $H_*(K) \to
H_*(K')$ is injective.
\end{lemm}
\begin{proof}
Let $L/k$ be a finitely generated field extension and $x \in L$ not a $p$-th
root. Write $L' = L(x^{1/p})$. I claim that $H_*(L) \to H_*(L')$ is injective.
Once this is done we conclude that $H_*(K) \to H_*(K')$ is injective for any
purely inseparable finitely generated extension (being a composite of finitely
many extensions of the form $L'/L$) and hence for any purely inseparable
extension by continuity. (See the beginning of Subsection
\ref{subsec:app-overview} for the continuous
extension of a homotopy module from a presheaf
on $Sm(k)$ to essentially smooth schemes.)

In order to prove the claim, we shall use the transfer $tr_{L'/L}: H_*(L') \to
H_*(L)$. This satisfies the projection formula: if $\alpha \in H_*(L)$ then
$tr_{L'/L} \alpha|_{L'} = tr_{L'/L}(1) \alpha$, where $1 \in \ul GW(L')$ is the
unit. (This is because transfer comes from an actual map of pro-spectra
$\Sigma^\infty Spec(L)_+ \to \Sigma^\infty Spec(L')_+$.)
Hence it is enough to show that $t := tr_{L'/L}(1)$ is a unit in
$GW(L)[1/p]$.  I claim that
\begin{equation} \label{eq:t-formula}
  t = \sum_{i=1}^p \langle (-1)^{i-1} \rangle.
\end{equation}
Indeed this may be checked by direct computation, using the fact that
(geometric) transfers on $\ul K_*^{MW}$ coincide with Scharlau transfers, as
follows from their definition \cite[Section 4.2]{A1-alg-top} and Scharlau's reciprocity law
\cite[Theorem 4.1]{scharlau1972quadratic}. We will explain this at the end of
the proof.

Next we need to show
that $t \in GW(L)[1/p]$ is invertible. For this, recall the fibre product
decomposition
\begin{equation*}
\begin{CD}
GW(L) @>{dim}>> \ZZ \\
@V{cl}VV        @VVV \\
W(L)  @>>>      \ZZ/2.
\end{CD}
\end{equation*}
Here $W(L) = GW(L)/h$ is the witt ring, $h = \langle 1, -1 \rangle$ is the
hyperbolic plane, $dim$ is the dimension homomorphism (determined by the
property that $dim(\langle a \rangle) = 1$ for all $a \in L^\times$) and $cl$ is the
canonical surjection (as is $\ZZ \to \ZZ/2$). Both injectivity and surjectivity of
the map $GW(L) \to
W(L) \times_{\ZZ/2} \ZZ$ follow from the fact that $GW(L) h = \ZZ h$
\cite[Lemma 1.16]{wittrings}.

Thus to show that $t \in GW(L)[1/p]$ is invertible,
it is enough to consider the
canonical images $dim(t) \in \ZZ[1/p]$ and $cl(t) \in W(L)[1/p]$. We know
that $dim(t) = p$ is invertible by design. If $p = 2$ then $W(L)[1/p] = 0$
\cite[Theorem III.3.6]{milnor1973symmetric}, so
$cl(t)$ is a unit. Otherwise we have that
$t = \frac{p-1}{2} (\langle 1 \rangle + \langle -1 \rangle) + 1$ and so
$cl(t) = 1$ is also invertible. This concludes
the proof, modulo the claim \eqref{eq:t-formula}.

\paragraph{Computation of $t$.}
Recall that
the transfer of a simple extension $GW(L') \to GW(L)$ is defined by considering the surjection
\[ K_1^{MW}(L(U)) \xrightarrow{\partial} \bigoplus_{P \in (\Aone_L)^{(1)}} GW(L[U]/P) =: R \to 0. \]
Now write $L' = L[U]/P_0$ for some irreducible polynomial $P_0$, pick $\alpha
\in GW(L')$, consider the
element $X \in R$ with $X_{P_0} = \alpha$ and $X_{P} = 0$ for $P \ne P_0$ and
pick a lift $X' \in K_1^{MW}(L(U))$. There exists a canonical boundary
$\partial^\infty: K_1^{MW}(L(U)) \to GW(L)$ and $tr(x) := \partial^\infty(X')$.
The fact that this is well-defined is not at all obvious and is discussed in
\cite[Section 4.2]{A1-alg-top}.

Let $f_P: L[U]/P$ be the $L$-linear map with $f_P(u^i) = 0$ for $i = 0, 1,
\dots, deg(P) -2$ and $f_P(u^{deg(P)-1}) = 1$. Here $u$ is the image of $U$.
Then given a bilinear space $B$ over
$L[U]/P$, we can view $B$ as a space over $L$ and transport the bilinear form to
have values in $L$ by applying $f_P$. The result is denoted $f_{P*} B$. This
extends by linearity to yield $f_{P*}: GW(L[U]/P) \to GW(L)$.

Let $Y = (Y_P)_P \in R$. Then the
Scharlau reciprocity theorem tells us that $\partial^\infty(Y) = \sum_P f_{P*}
\partial^P(Y)$. In particular, for $\alpha \in GW(L')$ we apply this to the lift
$X'$ to obtain $tr(x) = f_{P*}(x)$.

In our case, $x$ corresponds to the one-dimensional $L'$-vector space $V$ with basis $e$ and
bilinear form $B(ae, be) = ab$ for $a, b \in L'$. We have $P(U) = U^p - x$. Thus
as an $L$-vector space $V$ has basis $e, ue, \dots, u^{p-1} e$, where $u \in L'$
is the image of $U$. It follows that
\[ f_PB_V\left(\sum_i a_i u^i e, \sum_j b_j u^j e\right) = \sum_i a_i b_{p-1-i}. \]
Thus $f_{P*} V$ can be decomposed in the evident way into a sum
of two-dimensional bilinear spaces all of which are hyperbolic, and possibly a
single one-dimensional space with the standard form (this happens if and only if
$p$ is odd).
Hence we arrive at the
claimed formula $t = \sum_{i=1}^p \langle (-1)^{i-1} \rangle$.
\end{proof}

We note that Theorem \ref{thm:conservativity-I} can definitely fail if $k$
has infinite 2-étale
cohomological dimension. See \cite[Example 2.1.2(4)]{bondarko-effectivity} for
an example.

We also obtain the Pic-injectivity result:

\begin{thm}[Pic-injectivity] \label{thm:pic-injectivity}
Let $k$ be a perfect field of finite 2-étale
cohomological dimension and $E \in \SH(k)$ be invertible. If $ME \wequi \ZZ$
then $E \wequi S$.
\end{thm}
(Here $\ZZ, S$ denote the respective monoidal units of $\DM(k), \SH(k)$).
\begin{proof}
Invertible objects are compact because the unit is,
so we can use the conservativity
theorem. Since $\ZZ \in \DM(k)_{\ge 0}$ it follows from the ``more specifically'' part
that $E \in \SH(k)_{\ge 0}$. Let $a \in \Hom(\ZZ, ME)$ be an
isomorphism. We have $\Hom(\ZZ, ME) = \ul h_0(ME)_0(k) = (\ul
\pi_0(E)/\eta)_0(k) = \ul \pi_0(E)_0(k) / \eta \ul \pi_0(E)_1(k)$. It follows
that there exists $\tilde{a} \in \ul \pi_0(E)_0(k) = \Hom(S, E)$ such that
$M\tilde{a} = a$. Conservativity implies that $\tilde{a}$ is an isomorphism.
This concludes the proof.
\end{proof}

The remainder of this paper deals with the situation in which the 2-étale
cohomological dimension cannot
be finite, namely the orderable fields. We will not worry about
imperfectness problems, since we are really only interested in characteristic
zero. We begin with the following.

\begin{lemm}[Bondarko \cite{bondarko-effectivity}] \label{lemm:bondarko}
Let $k$ be of characteristic zero and finite virtual $2$-étale cohomological
dimension.

Let $E \in \SH(k)$ be connective and slice-connective, and suppose that $2^m
\id_E = 0$, for some $m \in \NN$. Then if $ME \in \DM(k)_{\ge n}$ also $E \in \SH(k)_{\ge n}$. In
particular if $ME = 0$ then $E=0$.
\end{lemm}
\begin{proof}
We may assume that $E \in \SH(k)_{\ge 0}$ and prove that
$\ul\pi_0(E)_* = 0$. By induction it will follow that $\ul\pi_i(E)_* = 0$ for
all $i$, and so $E = 0$ because the homotopy $t$-structure is non-degenerate.

Suppose that $E \in \SH(k)^{eff}(-N)$ (this is true for some $N$ because $E$ is
assumed slice-connective.)
By the same arguments as in the proof of Theorem \ref{thm:conservativity-I} it
suffices to show that for any finitely generated extension $K/k$ there exists $r
>> 0$ such that for any finite extension $L/K$ we have that $I(L)^r \ul
\pi_0(E)_{N+r}(L) = 0$.

Write $vcd_2$ for the
virtual $2$-étale cohomological dimension. Then $vcd_2(K) < \infty$ and
$vcd_2(L) \le vcd_2(K)$ \cite[Theorem 28 of Chapter 4]{shatz1972profinite}.
If $r > vcd_2(K) + m$ then
$I(L)^{r} = I(L)^{r-m}2^m$ \cite{elman-lum-2cohom}. But then $I(L)^r
\ul\pi_0(E)_{N+r}(L) = 0$ because $\ul\pi_0(E)_{N+r}(L)$ is a $2^m$-torsion
group. This concludes the proof.
\end{proof}

We will have to deal with several different monoidal categories at once. In
order to keep notations straight, we
will write $\tunit = \tunit_\mathcal{C}$ for the unit in any monoidal category
$\mathcal{C}$, omitting the reference to $\mathcal{C}$ if it is clear from
context. Thus for example $\tunit_{\SH(k)} = S$.

The unit $\tunit_{\SH(k)_2}$ is the spectrum $S_2 =
\hocolim(S \xrightarrow{2} S \xrightarrow{2} S \xrightarrow{2} \dots)$ and the
localisation $\SH(k) \to \SH(k)_2 \subset \SH(k)$ is given by smash product with
$S_2$.

Recall that the category $\SH(k)_2$ naturally splits into two parts. Indeed
$End(\tunit)_{\SH(k)_2} = GW(k)[1/2] = W(k)[1/2] \oplus \ZZ[1/2]$ \cite[Theorem
6.2.1, Remark 6.1.6 (b)]{morel2004motivic-pi0} \cite[Chapter 3 and Theorem
6.40]{A1-alg-top}. More specifically one
puts $\epsilon = -\langle -1 \rangle \in End(\tunit)$. Then $\epsilon^2 = 1$ and
hence $\epsilon_\pm = (\epsilon \pm 1)/2$ are idempotents and induce the
decomposition. We write $\SH(k)_2^+$ for those $E \in \SH(k)_2$ where $\epsilon$
acts as -1, and $\SH(k)_2^-$ for those $E$ where $\epsilon$ acts as +1; then
every object $E \in \SH(k)_2$ can be written uniquely as $E = E^+ \oplus E^-$,
with $E^\pm \in \SH(k)_2^\pm$. Note that $\eta[-1] = -\epsilon -1 = -2$ on
$\SH(k)_2^-$, so $\eta$ is invertible on $\SH(k)_2^-$.

\begin{corr} \label{corr:M-SH-}
Let $k$ be of characteristic zero and finite virtual 2-étale cohomological
dimension. If $E \in \SH(k)$ is connective and slice-connective, and $ME = 0$,
then $E \in \SH(k)_2^-$.
\end{corr}
\begin{proof}
Let $F$ fit in the distinguished triangle $E \xrightarrow{2} E \to F$. We have
$4\id_F = 0$ and so the above Lemma applies: $F = 0$. We conclude that $E
\xrightarrow{2} E$ is an isomorphism, whence $E \in \SH(k)_2$. Then $ME = M(E^+)
\oplus M(E^-)$. We have $M(E^-) = 0$ (on $UME^-$ the map $\eta$ is zero,
but it
is also an isomorphism since it comes from $\eta$ on $E^-$), thus $M(E^+) = 0$. But $M:
\SH(k)_2^+ \to \DM(k, \ZZ[1/2])$ is conservative on connective objects by
Corollary \ref{corr:hurewicz-conservativity} and Theorem \ref{thm:deglise} (note
that $\eta=0$ on $\ul K_*^{MW}[1/2]^- = \ul K_*^M[1/2]$). Thus $E^+ = 0$. This
concludes the proof.
\end{proof}

Note also that if $k$ is of arbitrary characteristic but
has finite 2-étale cohomological dimension then $k$ is
non-orderable, so $W(k)[1/2] = 0$ \cite[Theorem III.3.6]{milnor1973symmetric}
and $\SH(k)_2^- = 0$. Thus $\SH(k)_2 \to \DM(k, \ZZ[1/2])$ is conservative on
connective objects in this case. This finishes an argument in the proof of
Theorem \ref{thm:conservativity-II}.

\section{Witt Motives, Real Motives and the Real Étale Topology}
\label{sec:witt-motives}

We will now use the category of Witt motives \cite[Section 4]{levine2015witt}.
Recall that for any scheme $X$, one may define the Witt ring $W(X)$
\cite{knebusch-bilinear}. This is the ring of isometry classes of symmetric
bilinear forms on vector bundles on $X$, modulo hyperbolic bundles. The
associated sheaf $\ul{W}$ restricted to smooth varieties over a field is
unramified \cite[Theorem A]{panin-witt-purity}. This is the sheaf of unramified
Witt groups. One may check that it coincides with the sheaf defined by Morel
\cite[Example 3.34, Lemma 3.10]{A1-alg-top}.

To construct the category of Witt motives,
start with the category of complexes of
sheaves of $\ul W[1/2]$-modules. This category is abelian, so we may form its derived category
$\mathcal{D}$, which is even closed symmetric monoidal. Consider the
subcategory $\mathcal{E}$ of those objects $E$ such that the natural maps
$E \to \iHom(\Aone, E)$ (coming from pull-back) and $E
\to \iHom(\ZZ\Gm, E)$ (coming from the structure as a $\ul W$-module and
the identification $\ul W = \iHom(\ZZ\Gm, \ul W)$) are quasi-isomorphisms. This
is the category $\DM_W(k, \ZZ[1/2])$ of (2-inverted) Witt motives. We summarize
its basic properties.

\begin{lemm} \label{lemm:DMW-basic}
Let $k$ be a field of characteristic zero. The category $\DM_W(k, \ZZ[1/2])$ is
tensor triangulated and compact-rigidly generated. There is an adjunction
$F: \SH(k)_2^- \leftrightarrows \DM_W(k, \ZZ[1/2]): U$ with $F$ monoidal.
\end{lemm}
\begin{proof}
We can describe $\DM_W(k, \ZZ[1/2])$ as the homotopy category
a certain Bousfield localization of the
a model category of chain complexes of presheaves of $\ul{W}[1/2]$-modules. Let us
write $C(\ul{W}[1/2], k)$ for this model category of presheaves of $\ul{W}$-modules
and $LC(\ul{W}[1/2], k)$ for the Bousfield
localization such that $Ho(LC(\ul{W}, k)) = \DM_W(k,
\ZZ[1/2])$.

When using appropriate model structures on $C(\ul{W}[1/2], k)$ and
$\mathcal{SH}(k)$, there is a Quillen adjunction $\mathcal{SH}^{S^1}(k)
\leftrightarrows LC(\ul{W}[1/2], k)$. Passing to $\Gm$-spectra, we obtain a
Quillen adjunction $\mathcal{SH}(k) \leftrightarrows Spt(LC(\ul{W}[1/2], k), \Gm)$.
But in $\DM_W(k, \ZZ[1/2])$ the object corresponding to $\Gm$ is
$\otimes$-invertible by design, so $Ho(Spt(LC(\ul{W}[1/2], k), \Gm)) \wequi \DM_W(k,
\ZZ[1/2])$ \cite[Theorem 5.1]{hovey2001spectra}. We thus have an
adjunction $\SH(k) \leftrightarrows \DM_W(k, \ZZ[1/2]): U$. It is clear that
$U(\DM_W(k, \ZZ[1/2])) \subset \SH(k)_2^-$, so we have an adjunction as
displayed. It follows (since $k$ has characteristic zero) that $\DM_W(k,
\ZZ[1/2])$ is rigidly generated.

To prove that $\DM_W(k, \ZZ[1/2])$ is compactly generated note first that
$Ho(C(\ul{W}[1/2], k))$ is compactly generated (since we deal here with
\emph{presheaves} \cite[Example 1.5 and Theorem 5.4]{cisinski2009local}). I
claim that $LC(\ul{W}[1/2], k)$ can be obtained by inverting three kinds of
maps: $X \times \Aone \to X$, $X \to X \times \Gm$ (for $X \in Sm(k)$) and the
elementary Nisnevich squares. Here $X \to X \times \Gm$ comes from the point
-1 of $\Gm$. Once we know that it will follow that $\DM_W(k, \ZZ[1/2])$ is still
compactly generated. This is because we still have a ``bounded descent structure''
in the sense of \cite[Theorem 5.4]{cisinski2009local}.

The only part of the claim that needs proof is that inverting $X \to X \times
\Gm$ has the same effect as passing to complexes such that $\epsilon_E: E^\bullet
\to \iHom(\ZZ\Gm, E^\bullet)$. In order to understand this we may pass to
$Spt(C(\ul{W}[1/2], k))$ first, i.e. make $\Gm$ $\otimes$-invertible (indeed
both kinds of localizations make $\Gm$ $\otimes$-invertible). Then one sees that
inverting $X \to X \times \Gm$ corresponds to inverting the element $[-1] \in
[S^0, \Gm]$ whereas passing to complexes $E$ on which $\epsilon_E$ is an
equivalence means inverting $\eta \in [S^0, \Gm]$. However in $\SH(k)_2^-$ we
have the equality $\eta [-1] = -2$, whereas on $\SH(k)_2^+$ we have $2[-1]^2 =
0$ and so $[-1]$ is nilpotent (see \cite[Section 3.1]{A1-alg-top} for these
kinds of calculations). Thus inverting either element has the same
effect.
\end{proof}

We shall make good use of the next result, even though the proof is very easy. 
Recall that an object $E$ in a symmetric monoidal category $\mathcal{C}$ is
called \emph{rigid} if there exists an object $DE$ such that the functors
$\otimes E, \otimes DE$ are both left and right adjoint to one another. Rigid objects
are preserved by monoidal functors since they are detected by the zig-zag
equations \cite[Theorem 2.6]{MayBurnsideRing}. If $\mathcal{C}$ is a tensor
triangulated category then the subcategory $\mathcal{C}^{rig}$ of rigid objects
is a thick tensor triangulated subcategory \cite[Theorem A.2.5]{hovey1997axiomatic}.

\begin{lemm} \label{lemm:rigidity-tunit-trick}
Let $F: \mathcal{C} \leftrightarrows \mathcal{D}: G$ be an adjunction, where
$\mathcal{C}, \mathcal{D}$ are
symmetric monoidal triangulated categories.
Assume that $F$ is a symmetric monoidal, triangulated functor and that the
composite $GF$ preserves coproducts, that $\otimes_\mathcal{C},
\otimes_\mathcal{D}$ commute with coproducts (in both variables separately)
and that $\mathcal{C}$ is rigid-compactly generated.

Then for any $E \in \mathcal{C}$ there is a natural isomorphism $GFE \to E \otimes
G \tunit_\mathcal{D}$.
In particular if $\tunit_\mathcal{C} \to G\tunit_\mathcal{D}$ is an isomorphism,
then $F$ is fully faithful.
\end{lemm}
We note that for example if $\mathcal{D}$ is also
compact-rigidly generated, then $F$ preserves compact objects and $G$ preserves
coproducts, so $GF$ preserves coproducts. Moreover the tensor product commutes
with coproducts as seen as the monoidal structure is closed.

\begin{proof}
Let $E \in \mathcal{C}$. There is a natural map $GFE \to E \otimes G\tunit$
defined as follows. By adjunction and monoidality of $F$,
we only need $FE \to FE \otimes FG\tunit$. Now $FE \wequi FE \otimes \tunit$, so
it suffices to find $\tunit \to FG \tunit$. For this we use that $F\tunit \wequi
\tunit$, so we are trying to find $F\tunit \to FGF \tunit$, for which it
suffices by functoriality to find $\tunit \to GF \tunit$. This is clear by
adjunction.

Let $E \in \mathcal{C}$ be rigid. I claim that $GFE \to E \otimes
G\tunit$ is an isomorphism. To do this, let $T \in \mathcal{C}$. We compute
\[ [T, E \otimes G\tunit] \iso [T \otimes DE, G\tunit]
   \iso [FT \otimes DFE, \tunit] \iso [FT, FE] \iso [T, GFE]. \]
Here we have used that $F$ is monoidal and so preserves duals. The claim
follows by the Yoneda lemma.

Now let $A$ be the class of objects in $E \in \mathcal{C}$ such that $GFE \to E
\otimes G\tunit$ is an isomorphism. It is closed under isomorphisms and cones
(since $G, F$ are triangulated functors) and also under arbitrary sums (since
$GF$ preserves coproducts). Finally it contains all rigid objects, by the claim.
Consequently $A$ contains all objects of $\mathcal{C}$, by compact-rigid
generation.

For the last part, if $\tunit \to G\tunit$ is an isomorphism then it follows
that $GF \wequi \id$, and $F$ is fully faithful.
\end{proof}

This allows us to prove a somewhat deeper property of $\DM_W(k, \ZZ[1/2])$.

\begin{prop} \label{prop:DMW-t-tsructure}
The category $\DM_W(k, \ZZ[1/2])$ admits a $t$-structure such that in the
adjunction $F: \SH(k)_2^- \leftrightarrows \DM_W(k, \ZZ[1/2]): U$ the functor
$U$ is $t$-exact and induces an equivalence on the hearts.

In particular $F$ is conservative on connective objects.
\end{prop}
\begin{proof}
We will use again the notation of the proof of Lemma \ref{lemm:DMW-basic}.

We wish to use the obvious analogue of the homotopy $t$-structure. We shall use
Lemma \ref{lemm:transfer-$t$-structure}. In order to
do this we need to know that if $X \in Sm(k)$, then the fibrant replacement of
$\ul{W}[1/2] \otimes X \in C(\ul{W}[1/2], k)$ is non-negative in the standard
$t$-structure. The proof of Lemma \ref{lemm:DMW-basic} shows that it is enough
to know that $FUX \in \SH(k)_{\ge 0}$. But by Lemma
\ref{lemm:rigidity-tunit-trick} for this it is enough to know that $U\tunit \in
\SH(k)_{\ge 0}$.

The unit in $C(\ul{W}[1/2], k)$ is just the sheaf $\ul{W}[1/2]$. This is already
strictly homotopy invariant and $\epsilon$-local, so globally
equivalent to its own fibrant replacement! In particular it is non-negative.

This finishes the proof that $\DM_W(k, \ZZ[1/2])$ admits a homotopy
$t$-structure. By design $U$ is $t$-exact. If $F_* \in \SH(k)_2^{-, \heart}$
then $F$ consists of strictly invariant sheaves which are $\epsilon$-local, and
consequently $F$ defines an object $e(F) \in \DM_W(k, \ZZ[1/2])^\heart$. (All of
the sheaves $F_i$ are canonically isomorphic, so this is well defined.) The
functors $e$ and $U: \DM_W(k, \ZZ[1/2])^\heart \to \SH(k)_2^{-, \heart}$ are
manifestly inverse to each other.

The last sentence just follows from Corollary
\ref{corr:hurewicz-conservativity} and Lemma \ref{lemm:t-exactness-adjoints}.
\end{proof}

We now have to use the \emph{real étale topology} \cite[(1.2)]{real-and-etale-cohomology}.
To any scheme $X$ one may
functorially associate a topological space (in fact locally ringed
space) $R(X)$, called its ``realification''. If $X =
Spec(A)$ then one writes $Sper(A) := R(Spec(A))$. The points of $R(X)$ are
pairs $(P, \alpha)$ where $P \in X$ and $\alpha$ is an ordering of the residue
field $k(P)$. We call a family of étale maps $\{\alpha_i: X_i \to X\}$ a \emph{real étale
covering (or rét-covering)} if the induced maps $R(\alpha_i): R(X_i) \to R(X)$
are jointly surjective. For example, if $X = Spec(k)$ where $k$ is a field, then
$R(X)$ is the set of orderings of $k$, and a family of separable field extension
$l_i/k$ defines a rét-covering $\{Spec(l_i) \to Spec(k)\}$ if and only if every
ordering of $k$ extends to one of the $l_i$.

The rét-coverings define a topology on schemes
\cite[(1.2)]{real-and-etale-cohomology}. We write $a_{\ret}$ for the associated
sheaf functor, in various situations. Write $W$ for the presheaf of Witt
groups.

\begin{lemm}[\cite{karoubi2015witt}, Lemma 6.4 (2) and its proof]
The sheaf $a_{Nis} W[1/2]$ is a
sheaf in the rét-topology on $Sch/\ZZ[1/2]$: for a scheme $X$ such that $2$ is
invertible on $X$, the natural map
\[ (a_{Nis} W[1/2])(X) \to (a_{\ret} W[1/2])(X) \]
is an isomorphism. Moreover so is 
\[ (a_\ret \ZZ[1/2])(X) \to (a_{\ret} W[1/2])(X). \]
\end{lemm}

\begin{corr} \label{corr:W-ret-sheaf}
For any field $k$ of characteristic not two, the sheaf $\ul{W}[1/2]$ on $Sm(k)$
is a sheaf in the rét-topology. Moreover the natural map $\ZZ[1/2] \to
\ul{W}[1/2]$ induces an isomorphism $a_\ret \ZZ[1/2] = \ul{W}[1/2]$.
\end{corr}
\begin{proof}
The (Nisnevich!) sheaf of unramified Witt
groups $\ul{W}$ is just the sheaf associated (in the Zariski
topology) with the presheaf $W$ of Witt groups \cite[Theorem
A]{panin-witt-purity}.
\end{proof}

We can consider a rét-version of the category $\DM_W(k, \ZZ[1/2])$. This is
constructed in precisely the same way, just that we start with rét-sheaves
instead of Nisnevich sheaves. There is an evident adjunction $L: \DM_W(k, \ZZ[1/2])
\leftrightarrows \DM_W(k, \ZZ[1/2])^\ret: R$ induced by $a_\ret$ and the
forgetful functor. Note that since $\ul{W}[1/2] = a_\ret
\ZZ[1/2]$, the category $\DM_W(k, \ZZ[1/2])^\ret$ is really just built from the
derived category of chain complexes on $Shv(Sm(k)_\ret)$, the Witt notation is
mostly a coincidence.

\begin{lemm}
The canonical functor $L: \DM_W(k, \ZZ[1/2]) \to \DM_W(k, \ZZ[1/2])^\ret$ is an
equivalence of categories.
\end{lemm}
\begin{proof}
I claim that $\DM_W(k, \ZZ[1/2])^\ret$ is compactly generated. The category
$D(Shv(Sm(k)_\ret))$ is compactly generated by \cite[Proposition
1.1.9]{cisinski2013etale}, since the rét-cohomological dimension of a scheme $X$
is bounded by the Zariski dimension of $X$ \cite[Theorem 7.6]{real-and-etale-cohomology}.
The category $\DM_W(k, \ZZ[1/2])^\ret$ is a reflective localization of
$D(Shv(Sm(k)_\ret))$. In other words there is a fully faithful functor $\iota: \DM_W(k,
\ZZ[1/2])^\ret \to D(Shv(Sm(k)_\ret))$ which has a left adjoint $L'$. The
essential image of $\iota$ consists of the \emph{local objects}, i.e. those $E
\in D(Shv(Sm(k)_\ret))$ such that $\alpha^*: [T_2, E[i]] \to [T_1, E[i]]$ is an
isomorphism for all $i \in \ZZ$ and appropriate maps
$\alpha: T_1 \to T_2 \in D(Shv(Sm(k)_\ret))$. Namely, as in the proof of Lemma
\ref{lemm:DMW-basic}, the maps $\alpha$ we consider here are induced by
$X \times \Aone \to X$ and $X \to X \times \Gm$, for $X \in Sm(k)$. In
particular, these are maps between \emph{compact objects}. Since $\iota$ is fully faithful,
it is easy to see that the image under $L'$
of a generating set for $D(Shv(Sm(k)_\ret)$ is
a generating set for $\DM_W(k, \ZZ[1/2])^\ret$. Since local objects are
characterised in terms of maps out of certain compact objects,
they are closed under arbitrary sums. This implies that $\iota$ preserves
abitrary sums, and hence $L'$ preserves compact
objects. Thus the claim is proved.

It follows that $\DM_W(k, \ZZ[1/2])^\ret$ is compact-rigidly generated (since
the monoidal functor $L: \DM_W(k, \ZZ[1/2]) \to \DM_W(k, \ZZ[1/2])^\ret$ preserves
rigid objects) and so we may apply Lemma \ref{lemm:rigidity-tunit-trick}. It is
thus enough to show that $RL \tunit = \tunit$. But as we have said before the
sheaf $\ul{W}[1/2] \in LC(\ul{W}[1/2], k)$ is already fibrant, i.e.
$H^n_{Nis}(X \times \Aone, \ul{W}[1/2]) = H^n_{Nis}(X, \ul{W}[1/2]$ and
$H^n_{Nis}(X \times \Gm, \ul{W}[1/2]) = H^n_{Nis}(X, \ul{W}[1/2]_{-1}) =
H^n_{Nis}(X, \ul{W}[1/2]$. Since Nisnevich and rét-cohomology of rét-sheaves
agree \cite[Proposition 19.2.1]{real-and-etale-cohomology} and $\ul{W}[1/2]$
already \emph{is} a rét-sheaf by Corollary \ref{corr:W-ret-sheaf} we conclude
that $\ul{W}[1/2]$ is also fibrant as an object of $L^\ret C(\ul{W}[1/2], k)$,
i.e. in the model category underlying $\DM_W(k, \ZZ[1/2])^\ret$. Consequently
$R\tunit \wequi \tunit$ and we are done.
\end{proof}

As we have noted above, the category $\DM_W(k, \ZZ[1/2])^\ret$ is obtained as a
localization of $D(Shv(Sm(k)_\ret), \ZZ[1/2])$. We might thus attempt to use
Lemma \ref{lemm:transfer-$t$-structure} to build a $t$-structure on $\DM_W(k,
\ZZ[1/2])^\ret$ such that $\DM_W(k, \ZZ[1/2])^\ret \hookrightarrow
D(Shv(Sm(k)_\ret), \ZZ[1/2])$ is $t$-exact. The heart of this $t$-structure
will, by design, be a full subcategory of $Shv(Sm(k)_\ret)$. If this is possible,
we shall say that $\DM_W(k, \ZZ[1/2])^\ret$ admits a homotopy $t$-structure.
Note that we have already constructed a $t$-structure on the category $\DM_W(k,
\ZZ[1/2])$ which we now know is equivalent to $\DM_W(k, \ZZ[1/2])^\ret$.
However, it is not clear a priori that these two $t$-structures are the same.

\begin{lemm} \label{lemm:DMW-ret-t-structure}
The category $\DM_W(k, \ZZ[1/2])^\ret$ admits a homotopy $t$-structure, and the
equivalence $\DM_W(k, \ZZ[1/2])^\ret \wequi \DM_W(k, \ZZ[1/2])$ is $t$-exact.
\end{lemm}
\begin{proof}
We have ``forgetful'' functors
\[ D(Shv(Sm(k)_\ret), \ZZ[1/2]) \xrightarrow{U_1} D(\ul W[1/2]\Mod)
\xrightarrow{U_2} D(Shv(Sm(k)_{Nis})). \]
All three categories have canonical $t$-structures, and the left adjoints are
right $t$-exact. Moreover, $U_2$ is $t$-conservative (by definition). Together
with \cite[Proposition 19.2.1]{real-and-etale-cohomology}, this implies that
$U_1$ is $t$-exact. Consider the restricted functor $U_1:  D(Shv(Sm(k)_\ret),
\ZZ[1/2]) \supset \DM_W(k, \ZZ[1/2])^\ret \to \DM_W(k, \ZZ[1/2]) \subset D(\ul
W[1/2]\Mod)$. It follows from the above that
\begin{gather*} U_1(\DM_W(k, \ZZ[1/2])^\ret \cap D(Shv(Sm(k)_\ret),
\ZZ[1/2])_{\ge 0}) \subset \\ \DM_W(k, \ZZ[1/2])_{\ge 0} = \DM_W(k, \ZZ[1/2]) \cap
D(\ul W[1/2]\Mod)_{\ge 0},
\end{gather*}
and similarly for $\le 0$. Since $U_1: \DM_W(k, \ZZ[1/2])^\ret \to \DM_W(k,
\ZZ[1/2])$ is an equivalence, it follows that $\DM_W(k, \ZZ[1/2])^\ret \cap D(Shv(Sm(k)_\ret),
\ZZ[1/2])_{\ge 0}$ is the non-negative part of a (unique) $t$-structure on
$\DM_W(k, \ZZ[1/2])^\ret$. This proves simultaneously that the homotopy
$t$-structure exists and that the equivalence is $t$-exact.
\end{proof}

This result has the following interesting consequence. Recall the functor
$\Omega$ from the beginning of Section \ref{sec:homotopy-modules}.

\begin{corr} \label{corr:Nis-ret-sheaves}
Let $k$ be a field of characteristic zero.

Let $F$ be a Nisnevich sheaf of $\ul W[1/2]$-modules on $Sm(k)$ which is strictly homotopy
invariant and such that the natural map $F \to \Omega F$ is an isomorphism. Then
for any $p \ge 0$ and $X \in Sm(k)$ we have $H^p_{Nis}(X, F) = H^p_{\ret}(X, F)$.
\end{corr}
\begin{proof}
This follows from Proposition \ref{prop:DMW-t-tsructure} and Lemma
\ref{lemm:DMW-ret-t-structure}. Indeed the sheaves $F$ as in the corollary are
precisely objects in the heart of $\SH(k)_2^-$, which is the same as the
$\Aone$-invariant and $\Gm$-invariant rét-sheaves by the two results.
\end{proof}

From now on, we will identify $\DM_W(k, \ZZ[1/2])$ and $\DM_W(k,
\ZZ[1/2])^\ret$.

\begin{corr} \label{corr:W-basechange-cons}
Let $k$ have characteristic zero and $l_j/k$ be field extensions, $j \in
J$ some indexing set. Assume that every ordering of $k$ extends to an ordering
of one of the $l_j$, i.e. that $\{Spec(l_j) \to Spec(k)\}_j$ is a rét-covering.
Then the functor
\[ \DM_W(k, \ZZ[1/2]) \to \prod_{j \in J}\DM_W(l_j, \ZZ[1/2]) \]
is conservative.
\end{corr}
\begin{proof}
The base change functor is $t$-exact and so commutes with taking homotopy
sheaves. By the previous corollary all of the homotopy sheaves are sheaves in
the rét-topology, so it suffices to show that
\[ Shv_{\ret}(Sm(k)) \to \prod_j Shv_{\ret}(Sm(l_j)) \]
is conservative. This is essentially clear, since $\{Spec(l_j) \to Spec(k)\}_j$
is a rét-covering.
\end{proof}

The above corollary shows that for our purposes, in studying $\DM_W(k,
\ZZ[1/2])$ we may
assume that $k$ is real closed. It turns out that in this case the field is
essentially irrelevant. Before we can prove this, we need to recall a result
from semi-algebraic topology. This combines the work of several people.

\begin{thm}[Coste-Roy, Delfs, Scheiderer] \label{thm:real-closed-comparison}
Let $K/k$ be an extension of real closed fields and $X/k$ a separated scheme of finite
type. Then for any presheaf of abelian groups $F$ (on $X$) and any $p \ge 0$, the natural map
\[ H^p_{\ret}(X, F) \to H^p_{\ret}(X_K, F_K) \]
is an isomorphism, where $F_K$ denotes the pulled back sheaf.
\end{thm}
\begin{proof}
Associated with $X$ we have the real étale site $X_{\ret}$, the topological space
$R(X)$, and a further site which we shall denote $G(X)$ and which is called
the \emph{geometric space} associated with $X$ in \cite{delfs1991homology} (this reference
requires the finite type and separatedness assumptions). The toposes associated
with $X_{\ret}$ and $R(X)$ are equivalent \cite[Theorem
1.3]{real-and-etale-cohomology}, and the toposes associated with $G(X)$ and
$R(X)$ are also equivalent \cite[Proposition 1.4]{delfs1991homology}. Finally, the comparison
result is proved for cohomology in $G(X)$ in \cite[Theorem
6.1]{delfs1991homology} (take $\Phi = X$).
\end{proof}

\begin{prop} \label{prop:real-closed-comp}
Let $f: Spec(l) \to Spec(k)$ be an extension of real closed fields. The functor
$f^*: \DM_W(k, \ZZ[1/2]) \to \DM_W(l, \ZZ[1/2])$
is an equivalence of categories. Moreover, both categories are
generated by the monoidal unit.
\end{prop}
\begin{proof}
We first show that the functor is fully faithful. We wish to apply Lemma
\ref{lemm:rigidity-tunit-trick}. Certainly the functor $f^*$ is monoidal with
right adjoint $f_*$, and both $\DM_W(k, \ZZ[1/2])$ and $\DM_W(l, \ZZ[1/2])$ are
compact-rigidly generated (by Lemma \ref{lemm:DMW-basic}).
It remains to show that $f_*f^* \tunit \wequi
\tunit$. Since the $FX$ for $X \in Sm(k)$ generate $\DM_W(k, \ZZ[1/2])$, this
amounts to comparing $[FX[-p], \tunit] = H^p_\ret(X, \ZZ[1/2])$ and $[FX[-p], f_*f^*\tunit] =
[FX_l[-p], \tunit] = H^p_\ret(X_l, \ZZ[1/2])$. Thus we conclude using
Theorem \ref{thm:real-closed-comparison} and Corollary \ref{corr:Nis-ret-sheaves}

Next we need to show that the functor is essentially surjective. To do this it
suffices to show that $\DM_W(k,\ZZ[1/2])$ is generated by the monoidal unit $\tunit$.
For this it is enough to show that if $F \in \DM_W(k, \ZZ[1/2])^\heart$ is a sheaf
such that $F(k) = 0$ then $F=0$. Now $F$ is unramified, so it is enough to show
that $F(K) = 0$ for every finitely generated field extension $K/k$. But
Corollary \ref{corr:Nis-ret-sheaves} applies and thus $F$ is a rét-sheaf.
Hence it suffices to
show that $F(K^r) = 0$ for every real closure $f: Spec(K^r) \to Spec(K)$. But
this just says that $[f^* \tunit, f^* F] = 0$, which we know is true because $[f^*
\tunit, f^* F] = [\tunit, F] = F(k)$ by the fully faithfulness we already
proved.
\end{proof}

\begin{prop} \label{prop:W-real-closed-comparison}
The category $\DM_W(k, \ZZ[1/2])$ is canonically equivalent to $D(\ZZ[1/2])$,
where $k$ is any real closed field.
\end{prop}
\begin{proof}
There exists an adjunction of sites $e: \ZZ[1/2]\Mod \leftrightarrows
Shv(Sm(k)_\ret, \ZZ[1/2]): r$,
where $(rF) = F(Spec(k))$ and $e$ is the constant sheaf functor.
Composing with the localisation adjunction, we get $\tilde{e}: D(\ZZ[1/2]) \leftrightarrows
\DM_W(k, \ZZ[1/2])^\ret = \DM_W(k, \ZZ[1/2]): \tilde{r}$. Now $e \ZZ[1/2] =
a_\ret \ZZ[1/2]$ is already $\Aone$-$\Gm$-local, and so $\tilde{r} \tilde{e}
\tunit = \tunit$. By a further application of Lemma
\ref{lemm:rigidity-tunit-trick} we conclude that $\tilde{e}$ is fully faithful.
Proposition \ref{prop:real-closed-comp} implies that $\tilde{e}$ has dense image
and so is essentially surjective. This concludes the proof.
\end{proof}

Now let $k$ be an
arbitrary field (of characteristic zero) and $\sigma$ an ordering of $k$. Write
$k_\sigma$ for a real closure of $(k, \sigma)$. We denote by $M_\sigma[1/2]:
\SH(k) \to D(\ZZ[1/2])$ the composite $\SH(k) \to \SH(k_\sigma) \to
\DM_W(k_\sigma, \ZZ[1/2]) \wequi D(\ZZ[1/2])$ and call it the \emph{real motive
associated with $\sigma$}. We have thus proved the following:

\begin{thm}[Conservativity II] \label{thm:conservativity-II}
Let $k$ be a field of characteristic zero. Assume that $k$ has finite virtual
2-étale cohomological dimension. Write $Sper(k)$ for the set of orderings of $k$.

If $E \in \SH(k)$ is connective and slice-connective, and we have that $0 \wequi
ME \in \DM(k)$ and that for each $\sigma \in Sper(k)$, $0 \wequi M_\sigma[1/2](E)
\in D(\ZZ[1/2])$, then $E \wequi 0$.
\end{thm}
\begin{proof}
Combine Corollary \ref{corr:M-SH-}, Proposition \ref{prop:DMW-t-tsructure}, Corollary
\ref{corr:W-basechange-cons} and Proposition \ref{prop:W-real-closed-comparison}.
\end{proof}

\paragraph{Remark 1.} If $\sigma: k \subset \RR$ is an embedding of ordered
fields then one might define $M_\sigma'[1/2](E) \in D(\ZZ[1/2])$ as
$C_*(R_\RR(E), \ZZ[1/2])$, where $R_\RR$ denotes ``real realisation''
\cite[Section 3.3.3]{A1-homotopy-theory} and $C_*$ denotes the singular complex.
It is not hard to find a canonical equivalence $M_\sigma'[1/2](E) \wequi
M_\sigma[1/2](E)$, but we do not need this here.

\paragraph{Remark 2.} The author contends that the results in this section can be
proved more directly by using transfer considerations. For example to prove
Corollary \ref{corr:Nis-ret-sheaves}, it seems like a good first step to prove:
if $k$ is a field of characteristic zero and $H_*$ is a homotopy module which is
also a module over $\ul{W}[1/2]$ (i.e. $2$ and $\eta$ act invertibly on $H_*$),
then for any finitely generated field extension $K/k$ and any real étale cover
$\{Spec(L_i) \to Spec(K)\}_{i \in I}$, the map
\[ H_*(K) \to \prod_{i \in I} H_*(L_i) \]
is injective. Since $H_*$ has transfers satisfying the projection formula,
for this it is enough to prove that the ideal inside $W(K)[1/2]$ generated by
$tr_{L_i/K}(W(L_i)[1/2])$ is all of $W(K)[1/2]$. This can be checked directly,
using that $W(K)[1/2] \wequi \ZZ[1/2]^{Sper(K)}$ (and similarly for the $L_i$).

\emph{Added during revision:} this hope has now been fulfilled
\cite{bachmann-real-etale}.

\appendix
\section{Compact Objects in Abelian and Triangulated Categories}
\label{sec:compact-objects}

Here we collect some well-known but hard to reference facts about compact
objects. Recall that if $\mathcal{C}$ is a category with (countable) filtered
colimits then an object $X \in \mathcal{C}$ is called \emph{(countably) compact} if for
every (countable) filtering diagram $\{T_\lambda\}_{\lambda \in \Lambda}$ in
$\mathcal{C}$ the natural map $\colim_\lambda \Hom(X, T_\lambda) \to \Hom(X,
\colim_n T_\lambda)$ is an isomorphism.

Since triangulated categories do not usually have filtered colimits, this
definition would not make much sense. Instead an object $X$ in a triangulated
category $\mathcal{C}$ with (countable) coproducts is called (countably) compact
if for every (countable) family $\{T_\lambda\}_{\lambda \in \Lambda}$ (here $\Lambda$ is
just a set) the natural map $\bigoplus_\lambda \Hom(X, T_\lambda) \to \Hom(X,
\bigoplus_\lambda T_\lambda)$ is an isomorphism \cite[Definition 0.1]{Neeman1992}.

Finally recall sequential countable homotopy colimits in triangulated
categories. Suppose $\{T_n, f_n: T_n \to T_{n+1}\}_{n \in \NN}$ is a sequence in
the triangulated category $\mathcal{C}$. There is a natural map
\[ \id - s: \bigoplus_n T_n \to \bigoplus_n T_n,
   \text{``}(t_1, t_2, \dots) \mapsto (t_1, t_2 - f_1 t_1, t_3 - f_2 t_2, \dots).\text{''} \]
A cone on this map is denoted $\hocolim_n T_n$ \cite[Definition 1.4]{Neeman1992}.

Suppose $X$ is a countably compact object in the triangulated category $\mathcal{C}$ and
$\{T_n\}$ is a sequence. After choosing a presentation as
a cone, there is a natural map $\colim_n \Hom(X, T_n) \to \Hom(X, \hocolim_n T_n)$.
This map is an isomorphism \cite[Lemma 1.5]{Neeman1992}.

\begin{lemm} Let $\mathcal{C}$ be a $t$-category with countable coproducts and
such that $\mathcal{C}^\heart$ satisfies Ab5 and the inclusion
$\mathcal{C}^\heart \hookrightarrow \mathcal{C}$ preserves countable coproducts.
If $\{T_n\} \in
\mathcal{C}^\heart$ is a sequence of objects then there is an isomorphism
\[ \hocolim_n T_n \wequi \colim_n T_n. \]
Here $\hocolim_n T_n$ is computed in $\mathcal{C}$ by the above formula, and
$\colim_n T_n$ is computed in the abelian category $\mathcal{C}^\heart$.
\end{lemm}
We recall that Ab5 means that filtered colimits are exact. This is true in many
abelian categories, including $\HI_*(k) = \SH(k)^\heart$.
\begin{proof}
The point is that because of Ab5 the sequence
\[ 0 \to \bigoplus_n T_n \xrightarrow{\id - s} \bigoplus_n T_n \to \colim_n T_n \to 0 \]
is exact on the left (exactness at all other places follows from the definition
of colimit). In any $t$-category short exact sequences in the heart yield
distinguished triangles in the triangulated category, so we are done by the
definition of $\hocolim$.
\end{proof}

\begin{corr} In the above situation, if $E \in \mathcal{C}_{\ge 0}$ is countably compact
then so is $\pi_0^\mathcal{C} E \in \mathcal{C}^\heart$.
\end{corr}
\begin{proof}
Let $\{T_n\}$ be a sequence in $\mathcal{C}^\heart$. We compute
\begin{align*}
  \Hom(\pi_0^\mathcal{C} E, \colim_n T_n) &\stackrel{(1)}{=} \Hom(E, \colim_n T_n) \\
  &\stackrel{(2)}{=} \Hom(E, \hocolim_n T_n) \\
  &\stackrel{(3)}{=} \colim_n \Hom(E, T_n) \\
  &\stackrel{(4)}{=} \colim_n \Hom(\pi_0^\mathcal{C} E, T_n).
\end{align*}
This is the required result. (Here (1) is because $\colim_n T_n \in
\mathcal{C}_{\le 0}$, (2) is by the lemma, (3) is by compactness of $E$ and (4)
is because $T_n \in \mathcal{C}_{\le 0}$.)
\end{proof}

\section{Construction of $F_\bullet H$; General Case}
\label{sec:constructing-F*}

In this section $k$ is a perfect field; we shall prove Theorem
\ref{thm:slice-filtration} by constructing the relevant sheaves by hand and
establishing the necessary properties. We are trying to make this appendix
reasonably self-contained, in particular the first and introductory subsection.
There is no way around the fact however that it relies very heavily on the first
half of \cite{A1-alg-top}, so when proving technical results we will freely
reference that book.

\subsection{Overview of the Proof}
\label{subsec:app-overview}

Suppose for a moment that we can show (1) to (3). In Appendix
\ref{sec:compact-objects} we proved that if $E \in \SH(k)^{cpt}_{\ge 0}$ then
also $\ul \pi_0(E)_* \in \HI_*(k)$ is (countably) compact. (We also clarified the meaning of
compactness there.) Statement (4) then follows from the observation that
\[ H_* = \bigcup_n F_n H_*. \]
(Indeed $(F_n H)_n = H_n$ by definition.)

Thus we need only establish the following.

\begin{thm} \label{thm:slices-by-hand}
Let $k$ be a perfect field, $M_* \in \HI_*(k)$, and $n > 0, m \in
\ZZ$. Then there exists a strictly invariant sheaf $(\ul K_n^{MW} M_m)^{tr}$
with sections over fields given by formula \eqref{eq:tr-closure-formula}, and it
satisfies $\Omega(\ul K_n^{MW} M_m)^{tr} = (\ul K_{n-1}^{MW} M_m)^{tr}$.
(We put $(\ul K_0^{MW} M_m)^{tr} = M_m$.)
\end{thm}

In this section we shall outline the strategy, postponing the proofs of certain
technical lemmas to separate subsections at the end. As a very highlevel
overview, we will construct a presheaf $(\ul K_n^{MW} M_m)^{tr}$ which has the
correct sections on fields, show that it is strictly invariant, and finally
deduce that $\Omega(\ul K_n^{MW} M_m)^{tr} = (\ul K_{n-1}^{MW} M_m)^{tr}$.

Recall that if $F$ is a presheaf of sets on $Sm(k)$ and $X$ is an essentially
smooth scheme, i.e. the inverse limit in the category of schemes of a filtering
diagram of smooth schemes with affine transition morphisms $X \iso \lim_i X_i$,
then the colimit $F(X) := \colim_i F(X_i)$ is well-defined independent of the
presentation $X \iso \lim_i X_i$ \cite[Proposition 8.13.5]{EGAIV}. One thus obtains an extension of
$F$ to the category of essentially smooth schemes which is continuous in the
sense that it turns filtering inverse limits with affine transition morphisms
into colimits.

Recall also that a presheaf $F$ on $Sm(k)$ is called \emph{unramified} if (a) for
every $X \in Sm(k)$ with connected components $X_\alpha$ (i.e. $\alpha \in
X^{(0)}$) we have $F(X) = \prod_\alpha F(X_\alpha)$, (b) for every $U \subset X$
everywhere dense open ($X \in Sm(k)$) the pullback $F(X) \to F(U)$ is
injective, and finally (c) for every $X \in Sm(k)$ connected, the map $F(X) \to
\cap_{x \in X^{(1)}} F(\mathcal{O}_{X,x})$ is bijective \cite[Definition
2.1]{A1-alg-top}. Here the intersection is taken inside $F(k(X))$ (into which
each of the groups embeds by (b)). Note that we
are routinely being somewhat cavalier about the arguments of $F$; if $A$ is a
ring then we write $F(A) := F(Spec(A))$. Also if $X$ is a scheme and $x \in X$
is a point, we write $F(x) := F(k(x)) = F(Spec(k(x)))$, where $k(x)$ denotes the
residue field of $x$. This should always be unambiguous.

Write $\mathcal{F}_k$ for the category of field extensions of $k$ of finite
transcendence degree. Note that $\mathcal{F}_k^{op}$ is equivalent to a full
subcategory of the category of essentially smooth $k$-schemes. Hence any
presheaf of sets on $Sm(k)$ defines a continuous (covariant) functor
$F|_{\mathcal{F}_k}: \mathcal{F}_k \to Set$.

Suppose that $K \in \mathcal{F}_k$ and $\mathcal{O}_v \subset K$ is a discrete
valuation ring (dvr) with
residue field $\kappa(v)$. (We will always assume that $v|_k$ is trivial.) Then
we get a pullback $F(\mathcal{O}_v) \to F(K)$, as well as the specialisation $s:
F(\mathcal{O}_v) \to F(\kappa(v))$. If $F$ is invariant, then it is believable that this
data (for varying $K, \mathcal{O}_v$) almost determines $F$. One might then try
to define an unramified subpresheaf $I \subset F$ by specifying $I(K) \subset
F(K)$, putting $I(\mathcal{O}_v) = F(\mathcal{O}_v) \cap I(K)$, and trying to
restrict the specialisation maps. (If $F$ is
part of a homotopy module and $I$ is to be a homotopy sub-module, then one may
show that this definition of $I(\mathcal{O}_v)$ is the only one possible.) The
result below says that up to a compatibility condition, this strategy does
indeed work.

Recall the notion of an excellent ring \cite[Definition 2.35]{liu2002algebraic}.
Most rings which occur in
practice are excellent. Instead of considering all dvrs, it turns out that it
suffices to consider excellent drvs.

\begin{prop}[proof in Subsection \ref{subsec:constructing-strictly-inv}] \label{prop:constructing-I}
Let $M$ be a strictly invariant sheaf and $I: \mathcal{F}_k \to Ab$ a continuous
subfunctor of $M|_{\mathcal{F}_k}$. Suppose that for each $K \in \mathcal{F}_k$
and for each excellent dvr $\mathcal{O}_v \subset K$ we have that $s(I(K) \cap
M(\mathcal{O}_v)) \subset I(\kappa(v))$.

Then there exists a unique unramified sub-presheaf $I$ of $M$ extending $I:
\mathcal{F}_k \to Ab$ and with $I(\mathcal{O}_v) = I(K) \cap
M(\mathcal{O}_v)$ for each $K \in \mathcal{F}_k$ and dvr $\mathcal{O}_v \subset
K$. It is a Nisnevich sheaf and we have $I(X) = M(X) \cap I(X^{(0)})$, for all
$X \in Sm(k)$.
\end{prop}

The difficulty with applying this result is in showing that $s(I(K) \cap
M(\mathcal{O}_v)) \subset I(\kappa(v))$, since $M(\mathcal{O}_v)$ is not usually
very accessible. If $M$ is part of a homotopy module $M_*$, then the situation
is somewhat better. Namely if $K \in \mathcal{F}_k$, $\mathcal{O}_v \subset K$
is a dvr and $\pi$ is a uniformizer of $\mathcal{O}_v$, then there exists a
\emph{boundary morphism} $\partial^\pi: M_*(K) \to M_{*-1}(\kappa(v))$.
We may use the boundary to extend the specialisation maps to all of $M_*(K)$.
Recall that $M_*(K)$ is a module over $K_*^{MW}(K)$ and that every element
$a \in K^\times$ defines elements $\langle a\rangle \in K_0^{MW}(K), [a] \in K_0^{MW}(K)$.

\begin{prop}[proof in Subsection \ref{subsec:interaction-boundary-special}]
\label{prop:s-partial-formulas-firstpart}
Let $\mathcal{O}_v \subset K$ be a dvr with uniformizer $\pi$,
$M_*$ a homotopy module. Write $s: M_*(\mathcal{O}_v) \to M_*(\kappa(v))$ for
the specialisation and $\partial^\pi: M_*(K) \to M_{*-1}(\kappa(v))$ for the
boundary. Put \[s^\pi: M_*(K) \to M_*(\kappa(v)), s^\pi(m) =
\langle -1 \rangle \partial^\pi([-\pi]m).\]
Then for $m \in M_*(\mathcal{O}_v) \subset M_*(K)$ we have $s^\pi(m) = s(m)$
\end{prop}

It is thus important to analyse the interaction of the boundary maps with
transfers.

\begin{lemm}[proof in Subsection \ref{subsec:interaction-boundary-transfer}] \label{lemm:KnMmtr-bdry}
Let $\mathcal{O}_v \subset K$ be an excellent dvr with uniformizer $\pi$.
Then $\partial^\pi (\ul K_n^{MW} M_m)^{tr}(K) \subset (\ul K_{n-1}^{MW}
M_m)^{tr}(\kappa(v)) \subset M_{n+m-1}(\kappa(v))$, for each $n>0$.
\end{lemm}

\begin{corr}
The (so far hypothetical) sheaf $(\ul K_n M_m)^{tr} \subset M_{n+m}$
exists and is unramified.
\end{corr}
\begin{proof}
We want to apply Proposition \ref{prop:constructing-I}. For this we need to show
that if $\mathcal{O}_v \subset K$ is an excellent dvr then $s((\ul K_n^{MW} M_m)^{tr}(K)
\cap M_{n+m}(\mathcal{O}_v)) \subset (\ul K_n^{MW} M_m)^{tr}(\kappa(v))$. Using
Proposition \ref{prop:s-partial-formulas-firstpart}
we know that $s(x) = s^\pi(x) =
\langle -1 \rangle \partial^\pi([-\pi] x)$ and so $s((\ul K_n^{MW} M_m)^{tr}(K)
\cap M_{n+m}(\mathcal{O}_v)) \subset \partial^\pi([-\pi] (\ul K_n^{MW}
M_m)^{tr}(K)) \subset 
\partial^\pi((\ul K_{n+1}^{MW} M_m)^{tr}(K)) \subset (\ul K_n^{MW}
M_m)^{tr}(\kappa(v))$
by the lemma. (Here we are using the projection formula in the form $[-\pi]
tr(x) = tr([-\pi] x)$, see Subsection
\ref{subsec:interaction-boundary-transfer}.)
\end{proof}

Next we need to prove that the unramified Nisnevich sheaf $(\ul K_n M_m)^{tr}
\subset M_{n+m}$ is strictly homotopy invariant, i.e. that $H^p(X \times \Aone,
(\ul K_n M_m)^{tr}) \to H^p(X, (\ul K_n M_m)^{tr})$ is an isomorphism for all
$n$ and all $X \in Sm(k)$. A deep theorem of Morel says that it is only
necessary to prove this for $p=0, 1$. This together with some diagram chasing
yields the following general criterion.

\begin{lemm}[proof in Subsection \ref{subsec:constructing-strictly-inv}] \label{lemm:invariant-I}
Let $M$ be a strictly invariant sheaf and $I$ a subsheaf such that $I(X) = M(X)
\cap I(X^{(0)})$. Then $I$ is strictly invariant if and only if for all $L \in
\mathcal{F}_k$ we have $H^1(\mathbb{A}^1_L, I) = 0$.
\end{lemm}

In order to use this, we need to get down to the nitty gritty details of
constructing an actual resolution of $(\ul K_n M_m)^{tr}$ on $\Aone_L$. In order
to motivate this, recall that any homotopy module $M_*$ has a \emph{Rost-Schmid}
resolution \cite[Chapter 5]{A1-alg-top}. This is essentially built out of the boundary maps
$\partial^\pi$. Unfortunately these depend on the choice of uniformizer, which
is very inconvenient. It is possible to remove this problem by \emph{twisting}.
Let us discuss this technique first.

Recall that if $M$ is a
sheaf with a $\ZZ \Gm$-module structure and $\mathcal{L}$ is a line bundle on
$X$, then one writes $M(\mathcal{L})$ for the sheaf tensor product $M|_X
\otimes_{\ZZ \Gm}
\ZZ[\mathcal{L}^\times]$, where $\ZZ[\mathcal{L}^\times](U)$ is the free abelian
group on the trivialisations of $\mathcal{L}|_U$. The sections are denoted $M(U, \mathcal{L}) :=
M(\mathcal{L})(U)$. If $\mathcal{L}$ is trivial on the affine $U$, then $M(U, \mathcal{L}) =
M(U) \times_{\mathcal{O}^\times(U)} \mathcal{L}(U)^\times$, where
$\mathcal{L}(U)^\times$ is the set of generators of the free rank one
$\mathcal{O}(U)$-module $\mathcal{L}(U)$.
Thus $M(U, \mathcal{L})$ is isomorphic
to $M(U)$ for such $U$, but not canonically so.

If $M_*$ is a homotopy module then each $M_m$ is a $\ZZ \Gm$-module, via the
natural map $\ZZ \Gm \to \ul K_0^{MW}, a \mapsto \langle a \rangle$. This is the
only $\ZZ \Gm$-module structure we shall use for twisting. It follows that
$M_*(\mathcal{L}' \otimes \mathcal{L}^{\otimes 2}) \wequi M_*(\mathcal{L}')$,
canonically.

If $x \in X$, we put $\Lambda_x^X = \Lambda^{max}(m_x/m_x^2)$, where $m_x
\subset \mathcal{O}_{X, x}$ is the defining ideal and $\Lambda^{max}$ denotes
the maximal non-zero exterior power. This is a one-dimensional
$k(x)$-vector space.

Suppose that $X \in Sm(k)$ is connected and $x \in X^{(1)}$. Choose a
uniformizer $\pi$ of the dvr $\mathcal{O}_x$ (which we view as an element of
$m_x$) and define a homomorphism
\[ \partial_x: M_*(k(X)) \to M_{*-1}(k(x), \Lambda_x^X),
   m \mapsto m \otimes \pi. \]
It turns out that this is independent of the choice of $\pi$ \cite[Lemma 5.10]{A1-alg-top}. The
Rost-Schmid complex then starts as
\[ M_*(X) \to M_*(k(X)) \xrightarrow{\oplus \partial_x} \bigoplus_{x \in
      X^{(1)}} M_{*-1}(\kappa(x), \Lambda_x^X) \to \dots. \]
The next term involves $M_{*-2}(x)$ with $x \in X^{(2)}$ and so if $X$ has
dimension one, then the complex has to stop. It is thus plausible to try to
resolve $(\ul K_n^{MW} M_m)$ as follows:

\begin{equation} \label{eq:RS-complex}
0 \to (\ul K_n^{MW} M_m)^{tr}(X) \to (\ul K_n^{MW} M_m)^{tr} (X^{(0)})
    \xrightarrow{\partial}
    \bigoplus_{x \in X^{(1)}} (\ul K_{n-1}^{MW} M_m)^{tr}(k(x), \Lambda_x^X) \to 0.
\end{equation}

Certainly the first map is injective because $(\ul K_n^{MW} M_m)^{tr}$ is
unramified.
The boundary map $\partial$ is the same as for $M_{m+n}$.This makes sense
(i.e. the map lands in the group on the right) by Lemma \ref{lemm:KnMmtr-bdry}.
Since $M_*$ is a homotopy module its Rost-Schmid complex is a resolution.
Consequently in the complex \eqref{eq:RS-complex} the kernel of $\partial$ is
$M_{m+n}(X) \cap (K_n^{MW} M_m)^{tr} (X^{(0)})$. But by construction, i.e. the
last part of Proposition \ref{prop:constructing-I}, this is the same as $(\ul
K_n^{MW} M_m)^{tr}(X)$.

Note that the various complexes \eqref{eq:RS-complex} assemble to a complex of
presheaves on
$X_{Nis}$. Moreover, the  two resolving presheaves on the right are
acyclic Nisnevich sheaves \cite[Lemma 5.42]{A1-alg-top}. Thus the following
lemma will show that the Rost-Schmid complex is an acyclic resolution:

\begin{lemm}[proof in Subsection \ref{subsec:interaction-boundary-transfer}]
\label{lemm:RS-resn}
Let $X$ be the Henselization of a smooth variety over $k$ in a point of
codimension one. Then the differential
\[ (\ul K_n^{MW} M_m)^{tr} (X^{(0)})
    \xrightarrow{\partial}
    \bigoplus_{x \in X^{(1)}} (\ul K_{n-1}^{MW} M_m)^{tr}(k(x), \Lambda_x^X) \]
is surjective.
\end{lemm}

The next step is then to use this to show that $H^1(\Aone_K, (\ul K_n^{MW} M_m)^{tr})
= 0$. This is the content of the following lemma:

\begin{lemm}[proof in Subsection \ref{subsec:interaction-boundary-transfer}]
 \label{lemm:A1-diff-surjective}
Let $K \in \mathcal{F}_k$.
The differential
\[ (\ul K_n^{MW} M_m)^{tr}(K(T)) \to \bigoplus_{x \in
(\Aone_K)^{(1)}} (\ul K_{n-1}^{MW} M_m)^{tr}(k(x), \Lambda_x^{\Aone})\] is surjective.
\end{lemm}

We have thus established that $(\ul K_n^{MW} M_m)^{tr} \subset M_{m+n}$ is a
strictly homotopy invariant subsheaf. It remains to show that we have
constructed a homotopy module, i.e. that $\Omega (\ul K_n^{MW} M_m)^{tr} = (\ul
K_{n-1}^{MW} M_m)^{tr}$. Note that both sides are canonically subsheaves of
$M_{m+n-1}$ so this makes sense. This is essentially formal:

\begin{corr}
We have $\Omega (\ul K_n^{MW} M_m)^{tr} = (\ul K_{n-1}^{MW} M_m)^{tr}$.
\end{corr}
\begin{proof}
If $F$ is any strictly homotopy invariant Nisnevich sheaf, then in \cite[Chapter
5]{A1-alg-top} there is constructed for any $X$ essentially $k$-smooth the
Rost-Schmid complex (resolution)
\[ 0 \to F(X) \to F(X^{(0)}) \to \bigoplus_{x \in X^{(1)}} (\Omega F)(x) \to \dots. \]
Applying this to $X = \Aone_K$ and using strict homotopy invariance shows that
$(\Omega F)(K) = \partial_0^T(F(K(T)))$, where $\partial_0^T$ is the boundary
map corresponding to the point $0 \in \Aone_K$.

This construction is functorial in $F$. Given a subsheaf $F \to G$, we obtain an injection
$\Omega F \to \Omega G$, and we find that $(\Omega F)(K) = \partial_0^T(F(K(T)))
\subset (\Omega G)(K)$, where $\partial_0^T$ can mean either the boundary in the
Rost-Schmid complex of $F$ or $G$, because they are compatible.

Thus if $K \in \mathcal{F}_k$ then $\Omega (\ul K_n^{MW} M_m)^{tr}(K)
  = \partial_0^T((\ul K_n^{MW} M_m)^{tr}(K(T))) = (\ul K_{n-1}^{MW}
M_m)^{tr}(K)$.
Indeed the first equality holds for every strictly invariant sheaf, as we just
explained, and the second holds by Lemma \ref{lemm:A1-diff-surjective}.

Thus $\Omega (\ul K_n^{MW} M_m)^{tr}$ and $(\ul
K_{n-1}^{MW} M_m)^{tr}$ are two strictly invariant subsheaves of $M_{m+n-1}$
with the same sections on fields and hence equal. (If $F$ is strictly invariant
so is $\Omega F$ \cite[Lemma 2.32 and Theorem 5.46]{A1-alg-top}.)
\end{proof}

All parts of Theorem \ref{thm:slices-by-hand} are now proved (modulo the
postponed technical points).

\subsection{An Application}
We cannot resist
giving the following application. It is not used in the rest of the text.

\begin{corr}[compare \cite{rigidity-in-motivic-homotopy-theory}]
\label{corr:rigidity}
Let $M_*$ be a homotopy module which is torsion of exponent prime
to the characteristic of $k$, and $L/K$ an extension of algebraically closed
fields. Then $M_*(L) = M_*(K)$.
\end{corr}
\begin{proof}
We shall call any sheaf $F$ with $F(L) = F(K)$ rigid. Rigid sheaves are stable
under extensions, limits and colimits.

Let $K_n = ker(\eta^n: M_{*+n} \to M_*)$. Then there are exact sequences $0 \to
K_n \to K_{n+1} \to K_{n+1}/K_n \to 0$. Since $\eta$ acts trivially on $K_{n+1}/K_n$,
this homotopy module has transfers by Theorem
\ref{thm:deglise} and so is rigid \cite[Theorem 7.20]{lecture-notes-mot-cohom}.
Also $\eta$ acts trivially on $K_1$, so this
module is rigid by the same argument. By induction we find that $K_n$ is rigid
for all $n$ (recall that rigid modules are stable under extensions), and hence
$K_\infty := \colim_n K_n$ is rigid. It is thus sufficient to show that
$M/K_\infty$ is rigid. We may thus assume that $\eta: M_{*+1} \to M_*$ is
\emph{injective}.

Next consider the exact sequences $0 \to K^{(n)} :=
\eta^n(M_{*+n})/\eta^{n+1}(M_{*+n+1}) \to M_*/\eta^{n+1}M_{*+n+1} \to
M_*/\eta^n M_{*+n} \to 0$. By induction on $n$ (starting at $n=0$) using the
fact that $\eta$ acts trivially on $K^{(n)}$ and so $K^{(n)}$ is rigid as
before, one finds that $M_*/\eta^n M_{*+n}$ is rigid for all $n$. Let
$\eta^\infty(M)_* = \cap_n \eta^n M_{*+n}$. It follows that
$M_*/\eta^\infty(M)_*$ is rigid, being an inverse limit of rigid sheaves. It
remains to show that $\eta^\infty(M)$ is rigid.

Now $\eta$ was assumed injective on $M$. From this it follows that $\eta$ is an
isomorphism on $\eta^\infty(M)$. So in particular from now on we may assume that $\eta$ is
\emph{surjective} on $M$.

Consider $\tilde{F}_1M/\tilde{F}_0M =: M'$. It is straightforward to check that $\eta$
acts trivially on this module, so it is rigid. On the other hand
$(\tilde{F}_1M)_1(K) = M_1(K)$ whereas $(\tilde{F}_0M)_1(K) = K_1^{MW}(K)M_0(K)
= K_1^{MW}(K)\eta M_1(K) = I(K) M_1(K) = 0$, since $K$ is algebraically closed.
Thus $M'(K)_1 = M(K)_1$ and similarly for $L$, so we conclude that $M_1$ is
rigid. The general result follows by replacing $M_*$ by $M_{*+n}$.
\end{proof}

\subsection{Constructing Strictly Invariant Subsheaves: Proofs of Proposition
   \ref{prop:constructing-I} and Lemma \ref{lemm:invariant-I}}
\label{subsec:constructing-strictly-inv}

\begin{prop*}[\ref{prop:constructing-I}]
Let $M$ be a strictly invariant sheaf and $I: \mathcal{F}_k \to Ab$ a continuous
subfunctor of $M|_{\mathcal{F}_k}$. Suppose that for each $K \in \mathcal{F}_k$
and for each excellent dvr $\mathcal{O}_v \subset K$ we have that $s(I(K) \cap
M(\mathcal{O}_v)) \subset I(\kappa(v))$.

There exists a unique unramified sub-presheaf $I$ of $M$ extending $I:
\mathcal{F}_k \to Ab$ and with $I(\mathcal{O}_v) = I(K) \cap
M(\mathcal{O}_v)$ for each $K \in \mathcal{F}_k$ and excellent dvr $\mathcal{O}_v \subset
K$. It is a Nisnevich sheaf and we have $I(X) = M(X) \cap I(X^{(0)})$.
\end{prop*}
\begin{proof}[Proof of Proposition \ref{prop:constructing-I}.]
We use \cite[Theorem 2.11]{A1-alg-top}. This says that the functor $I:
\mathcal{F}_k \to Ab$ together with the data $I(\mathcal{O}_v)$ and the
specialisations inherited from $M$ extends to a (unique) unramified presheaf,
provided the data satisfies certain axioms called (A1) to (A4). In fact the
theorem establishes an equivalence of categories, so $I$ will be a subpresheaf
of $M$. Also while it is not stated in the source, the proof of the theorem only
ever uses the value of $I(\mathcal{O}_v)$ on geometrically constructed dvrs,
i.e. localisations of finitely generated rings over a field, and such rings are
excellent \cite[Theorem 2.39(a, c)]{liu2002algebraic}. Hence we only need the data $I(\mathcal{O}_v)$ for
excellent dvrs.

It is easy to check that the axioms (A1) to (A4) for
$I$ are all direct consequences of the axioms for $M$ and our definitions, so
we know that $I$ exists as a presheaf, and that it has the correct sections over
fields and over dvrs.

The formula $I(X) = I(X^{(0)}) \cap M(X)$ for all $X$ follows from the formula
being true for $X = Spec(\mathcal{O}_v)$, because $I$ is unramified. Indeed we
compute for $X$ connected
\begin{gather*}
  I(X) = \bigcap_{x \in X^{(1)}} I(\mathcal{O}_{X,x})
       = \bigcap_{x \in X^{(1)}} (I(k(X)) \cap M(\mathcal{O}_{X,x})) \\
        = I(k(X)) \cap \bigcap_{x \in X^{(1)}} M(\mathcal{O}_{X,x})
        = I(k(X)) \cap M(X).
\end{gather*}
(The first equality is because $I$ is unramified, the second is by definition,
and the fourth is because $M$ is unramified.) The case of non-connected $X$
follows directly from this because for an unramified presheaf $F$ we have $F(X
\coprod Y) = F(X) \oplus F(Y)$, by definition.

To prove that $I$ is a sheaf, let $\{U_\alpha\}_\alpha$ be a Nisnevich cover of
$X$ and $\{i_\alpha\}_\alpha$ a compatible family, $i_\alpha \in I(U_\alpha)$.
Then there exists a unique
``glued'' section $i \in M(X)$ with $i|_{U_\alpha} = i_\alpha$. We need to show that
$i \in I(X)$, i.e. that $i \in I(X^{(0)})$. Let $\xi \in X^{(0)}$. Since $\{U_\alpha\}_\alpha$ is a
Nisnevich cover there exists $\alpha_\xi$ such that $\xi \in U_{\alpha_\xi}$. It
then follows that $i|_\xi = (i|_{U_\xi})|_\xi = i_{\alpha_\xi}|_\xi \in I(\xi)$.
This concludes the proof.
\end{proof}

\begin{lemm*}[\ref{lemm:invariant-I}]
Let $M$ be a strictly invariant sheaf and $I$ a subsheaf such that $I(X) = M(X)
\cap I(X^{(0)})$ for all $X$.
Then $I$ is strictly invariant if and only if for all $L \in
\mathcal{F}_k$ we have $H^1(\mathbb{A}^1_L, I) = 0$.
\end{lemm*}
\begin{proof}[Proof of Lemma \ref{lemm:invariant-I}.]
We note that an arbitrary presheaf $F$ is $\Aone$-invariant if and only if for each $X$,
the map $F(X \times \Aone) \to F(X)$ induced by the zero section $X \to X \times
\Aone$ is injective. (Pullback along the zero section shows that $F(X)$
is a retract of $F(X \times \Aone)$.)
In particular $I$ is invariant, being a subsheaf of an
invariant sheaf.

By \cite[Theorem 5.46]{A1-alg-top} we then know that $I$ is strictly invariant
if and only if $H^1(\bullet, I)$ is invariant. Let $C=M/I$ be the Nisnevich
quotient sheaf and consider the following commutative diagram with exact rows
induced by the zero section $X \to \Aone \times X$.
\begin{figure*}[h]
\centering
\begin{tikzcd}[column sep=tiny]
H^0(\Aone \times X, I) \arrow{r} \arrow{d}{\wequi} & H^0(\Aone \times X, M) \arrow{r} \arrow{d}{\wequi}
  & H^0(\Aone \times X, C) \arrow{r} \arrow{d}{\alpha}
  & H^1(\Aone \times X, I) \arrow{r} \arrow{d}{\beta} & H^1(\Aone \times X, M) \arrow{d}{\wequi} \\
H^0(X, I) \arrow{r} & H^0(X, M) \arrow{r} & H^0(X, C) \arrow{r} & H^1(X, I) \arrow{r} & H^1(X, M)
\end{tikzcd}
\end{figure*}
A chase reveals that $\alpha$ is injective if and only if $\beta$ is injective.
Hence $I$ is strictly invariant if and only if $C$ is invariant.

Now let $X$ be essentially smooth over $k$ and Nisnevich local (i.e. the
spectrum of a Henselian ring) with fraction field $K$, and consider the
following commutative diagram with exact rows.
\begin{equation*}
\begin{CD}
  I(X) @>>> M(X) @>>> C(X) @>>> 0 \\
 @VVV       @VVV      @VV{\gamma}V \\
  I(K) @>>> M(K) @>>> C(K) @>>> 0
\end{CD}
\end{equation*}
Since the left square is Cartesian by the assumption that $I(X) = M(X) \cap
I(K)$, the morphism $\gamma$
is injective. Now let $X$ be arbitrary. Since $C$ is a Nisnevich sheaf, we
conclude that the composite $C(X) \to \bigoplus_{x \in X}C(\mathcal{O}_{X,x}^h)
\to \bigoplus_{x \in X}C(Frac(\mathcal{O}_{X,x}^h))$ is injective, and hence so
is $C(X) \to \bigoplus_{F \to X} C(F)$. Here the sum is over all fields $F \in
\mathcal{F}_k$ and all $F$-points of $X$. Now consider the following commutative
diagram.
\begin{equation*}
\begin{CD}
C(X \times \Aone) @>{\alpha}>> \bigoplus_{F \to X}C(F \times \Aone) @>>> \bigoplus_{F \to X \times \Aone} C(F) \\
@VVV                             @VVV \\
C(X) @>>> \bigoplus_{F \to X}C(F)
\end{CD}
\end{equation*}
The top right map comes from the fact that $\Hom(Spec(F), X \times \Aone) =
\Hom(Spec(F), X) \times \Hom(Spec(F), \Aone)$. By what we just said the top
composite is injective and hence so is $\alpha$. Thus the left vertical map is
injective if the right one is. That is, $C$ is invariant if (and only if) $C(F \times \Aone)
\to C(F)$ is injective for all $F \in \mathcal{F}_k$, which (going back to the
first diagram) is the same as $H^1(F \times \Aone, I) = 0$. This concludes the
proof.
\end{proof}

\subsection{Interaction of Boundary and Specialisation: Proof of Proposition
  \ref{prop:s-partial-formulas-firstpart}}
\label{subsec:interaction-boundary-special}

We will actually prove the following more general statement, which we need later
anyway. Recall that for $a \in K^\times$ there are elements $\langle a \rangle
\in K_0^{MW}(K), [a] \in
K_1^{MW}(K)$ \cite[Section 3.1]{A1-alg-top}.

\begin{prop} \label{prop:s-partial-formulas}
Let $\mathcal{O}_v \subset K$ be a dvr with uniformizer $\pi$,
$M_*$ a homotopy module. Write $s: M_*(\mathcal{O}_v) \to M_*(\kappa(v))$ for
the specialisation and $\partial^\pi: M_*(K) \to M_{*-1}(\kappa(v))$ for the
boundary. Put \[s^\pi: M_*(K) \to M_*(\kappa(v)), s^\pi(m) =
\langle -1 \rangle \partial^\pi([-\pi]m).\] Then the following hold.
\begin{enumerate}[(i)]
\item For $m \in M_*(\mathcal{O}_v)$ we have $s^\pi(m) = s(m)$.
\item For $\alpha \in \ul K_*^{MW}(K), m \in M_*(K)$ we have
      $s^\pi(\alpha m) = s^\pi(\alpha) s^\pi(m)$.
\item For $\alpha \in \ul K_*^{MW}(K), m \in M_*(K)$ we have
      \[ \partial^\pi(\alpha m) = \partial^\pi(\alpha) s^\pi(m) + s^\pi(\alpha)
         \partial^\pi(m) + \partial^\pi(\alpha) [-1] \partial^\pi(m). \]
\end{enumerate}
\end{prop}
\begin{proof}
We begin with two preliminaries. Firstly, if $M_* = \ul K_*^{MW}$, then all of
the relations hold by \cite[Lemma 3.16]{A1-alg-top}.

Secondly, suppose that
$\alpha \in \ul K_*^{MW}(\mathcal{O}_v)$ and $m \in M_*(K)$ (respectively
$\alpha \in \ul K_*^{MW}(K)$ and $m \in M_*(\mathcal{O}_v)$). Then we have
$\partial^\pi(\alpha m) = s(\alpha) \partial^\pi(m)$ (respectively
$\partial^\pi(\alpha m) = \partial^\pi(\alpha) s(m)$). To see this, one goes
back to the geometric construction of $\partial^\pi$ as in \cite[Corollary
2.35]{A1-alg-top}. That is we observe that after canonical identifications,
$\partial^\pi$ corresponds to the boundary map $\partial: H^0(K, M_*) \to
H^1_v(\mathcal{O}_v, M_*)$ in the long exact sequence for cohomology with
support. Our claim then follows from the observation that for any sheaf of rings
$K$ and $K$-module $M$, the boundary map in cohomology with support satisfies
our claim. This is easy to prove using e.g. Godement resolutions.

Statement (i) follows now from $\partial^\pi([-\pi]) = \langle -1 \rangle$.

We can prove statements (ii) and (iii) most quickly together. To do this, let $\ul
K_*^{MW}(K)[\xi]$ be the graded-$\epsilon$-commutative ring with $\xi$ in degree
one, $\xi^2 =
\xi[-1]$ and let $M_*(K)[\xi] := M_*(K) \oplus M_{*-1}(K)\xi$ be the evident $\ul
K_*^{MW}(K)[\xi]$-module. Put $\Theta: M_*(K) \to M_*(\kappa(v))[\xi],
\Theta(a) = s^\pi(a) + \partial^\pi(a) \xi$. (See \cite[Lemma
3.16]{A1-alg-top}.) Then (ii) and (iii) are equivalent to $\Theta(\alpha m) =
\Theta(\alpha) \Theta(m)$. It is sufficient to prove this on generators $\alpha$
of $\ul K_*^{MW}(K)$. Moreover (ii) and (iii) for $\alpha \in
K_*^{MW}(\mathcal{O}_v)$ hold by the second preliminary point. It follows from
\cite[Theorem 3.22, Lemma 3.14]{A1-alg-top} and standard formulas in Milnor-Witt
K-theory that $\ul K_*^{MW}(K)$ is generated by $\ul K_*^{MW}(\mathcal{O}_v)$
together with $[\pi]$. Hence to establish (ii) and (iii) we need only consider
$\alpha = \pi$.

These are easy computations: $s^\pi([\pi] m) = \partial^\pi([-\pi] [\pi] m) = 0$
because $[-\pi] [\pi] = 0$ \cite[Lemma 3.7 1)]{A1-alg-top} and also $s^\pi([\pi]) =
0$, so $s^\pi([\pi]) s^\pi(m) = 0 = s^\pi([\pi]m)$ as desired. Similarly
$\partial^\pi([\pi] m) = \partial^\pi(([-1] + \langle -1 \rangle [-\pi])m) =
[-1] \partial^\pi (m) + \langle -1 \rangle \partial^\pi([-\pi]m)$ and this is
the same as $\partial^\pi([\pi]) s^\pi(m) + s^\pi([\pi]) \partial^\pi(m) +
\partial^\pi([\pi]) [-1] \partial^\pi(m)$ because $\partial^\pi([\pi]) = 1$ and
$s^\pi([\pi]) = 0$.
\end{proof}

\subsection{Interaction of Boundary and Transfer: Proof of Lemmas
\ref{lemm:KnMmtr-bdry}, \ref{lemm:RS-resn} and \ref{lemm:A1-diff-surjective}}
\label{subsec:interaction-boundary-transfer}

Finally we come to transfers. Write $\omega$ for the canonical sheaf. We will
use it for twisting. By
\cite[Chapter 5]{A1-alg-top}, there are \emph{absolute transfers} $Tr^L_K: M_*(L,
\omega) \to M_*(K, \omega)$ for any finite extension $f: Spec(L) \to Spec(K)$ in $\mathcal{F}_k$.
These satisfy many properties, for example the projection formula
$Tr^L_K(f^*(\alpha) m) = \alpha Tr^L_K(m)$ whenever $\alpha \in \ul K_*^{MW}(K),
m \in M_*(L)$.
It follows that for a line bundle $\mathcal{L}$ on $K$ there exist natural
transfers $M_*(L, f^*(\mathcal{L}) \otimes \omega) \to M_*(K, \mathcal{L}
\otimes \omega)$. Put $\nu(f) = \omega \otimes f^*(\omega)^{-1}$. Then we find
the twisted transfer $M_*(L, \nu(f)) \to M_*(K)$. We shall need
the following result, which says that the transfers on fields extend to
arbitrary finite morphisms.

\begin{thm}[\cite{A1-alg-top}, Corollary 5.30] \label{thm:comm-tr}
Let $f: X' \to X$ be a finite morphism of essentially $k$-smooth schemes. Then
there is a commutative diagram
\begin{equation*}
\begin{CD}
M_*(X', \nu(f)) @>>> M_*(X'^{(0)}, \nu(f)) @>>> \bigoplus_{x \in X'^{(1)}} M_{*-1}(k(x), \nu(f) \otimes \Lambda_x^{X'}) \\
@VVV                   @VV{Tr}V                       @VV{Tr}V \\
M_*(X)          @>>> M_*(X^{(0)})          @>>> \bigoplus_{x \in X^{(1)}} M_{*-1}(k(x), \Lambda_x^X).
\end{CD}
\end{equation*}
\end{thm}
Here the middle vertical arrow is the transfer just defined. If $x \in X'^{(1)}$
then we have $\nu(f) \otimes \Lambda_x^{X'} \wequi \omega \otimes
f^*(\omega_X)^{-1}$ and hence there is a natural transfer $M_{*-1}(k(x), \nu(f)
\otimes \Lambda_x^{X'}) \to M_{*-1}(k(f(x)), \Lambda_{f(x)}^X)$. The right
vertical map is just the evident sum of such transfers. The left horizontal maps
are injective, so the theorem asserts that the right square is commutative and
that the unique vertical map on the left making the diagram commutative exists.

We now put \[(\ul K_n^{MW} M_m)^{tr}(K) = \left\langle Tr^L_K (\ul
K_n^{MW}(L, \nu(L/K)) M_m(L)) \right\rangle_{L/K}.\]
This is the ``correct'' definition of
$(\ul K_n^{MW} M_m)^{tr}(K),$ using twisted absolute transfers.
But note that if we just choose some isomorphisms
$\omega_K \wequi K$ and $\omega_L \wequi L$, we get the composite
$K_n(L)M_m(L) \subset M_{n+m}(L) \wequi M_{n+m}(L, \omega_L) \xrightarrow{Tr} M_{m+n}(K,
\omega_K) \wequi M_{m+n}(K)$ and this has the same image as the ``correct''
transfer, so the twists do not actually matter in some sense. (Note also that in
the definition of $(\ul K_n^{MW} M_m)^{tr}(K)$, it makes no difference if we
view the right hand side as the subgroup generated or the sub-$GW(K)$-module
generated. This follows from the projection formula.)

\begin{lemm*}[\ref{lemm:KnMmtr-bdry}]
Let $\mathcal{O}_v \subset K$ be an excellent dvr with uniformizer $\pi$.
Then $\partial^\pi (\ul K_n^{MW} M_m)^{tr}(K) \subset (\ul K_{n-1}^{MW}
M_m)^{tr}(\kappa(v)) \subset M_{n+m-1}(\kappa(v))$, for each $n>0$.
\end{lemm*}
\begin{proof}[Proof of Lemma \ref{lemm:KnMmtr-bdry}.]
Let $f: Spec(L) \to Spec(K)$ be a finite extension and $A$ the integral closure of
$\mathcal{O}_v$ in
$L$. Then $A$ is semilocal, smooth over $k$ (indeed $A$ is normal and
one-dimensional) and finite over
$\mathcal{O}_v$ \cite[Theorem 2.39(d)]{liu2002algebraic}. Let $x_1, \dots,
x_n$ be the closed points. From Theorem
\ref{thm:comm-tr} we get the commutative diagram
\begin{equation*}
\begin{CD}
M_{n+m}(L, \nu(f)) @>>> \bigoplus_i M_{n+m-1}(k(x_i), \Lambda_{x_i}^A \otimes \nu(f)) \\
@VV{Tr}V            @VV{\Sigma_i Tr}V \\
M_{n+m}(K) @>>> M_{m+n-1}(\kappa(v), \Lambda_v^{\mathcal{O}_v}).
\end{CD}
\end{equation*}
The choice of $\pi$ trivialises $\Lambda_v^{\mathcal{O}_v}$ and the bottom
horizontal arrow becomes $\partial^\pi$. Let $\alpha \in \ul K_n^{MW}(L, \nu(f))$ and $m
\in M_m(L)$. Then $\partial^\pi(Tr(\alpha m)) = \Sigma_i Tr(\partial_i(\alpha))
m)$ and we see that it is enough to prove that $\partial_i(\ul K_n^{MW}(L, \nu(f)) M_m(L))
\subset \ul K_{n-1}^{MW}(k(x_i), \Lambda_{x_i}^A \otimes \nu(f)) M_m(k(x_i))$.
We note that $\nu(f)$ is (non-canonically) trivial so we can ignore that twist.

We may thus forget about the extension $L/K$ and just prove that
$\partial^\pi(\ul K_n^{MW}(K) M_m(K)) \subset K_{n-1}(K) M_m(K)$ (renaming $L$
to $K$). But this
follows from Proposition \ref{prop:s-partial-formulas} part (iii) and the fact
that $\ul K_{> 0}^{MW}(K)$ is generated in degree 1.
\end{proof}

\begin{lemm*}[\ref{lemm:RS-resn}]
Let $X$ be the Henselization of a smooth variety over $k$ in a point of
codimension one. Then the differential
\[ (\ul K_n^{MW} M_m)^{tr} (X^{(0)})
    \xrightarrow{\partial}
    \bigoplus_{x \in X^{(1)}} (\ul K_{n-1}^{MW} M_m)^{tr}(k(x), \Lambda_x^X) \]
is surjective.
\end{lemm*}
\begin{proof}[Proof of Lemma \ref{lemm:RS-resn}.]
Let $X$ have generic point $\xi$ and closed point $x$.
We need to prove that $\partial: (\ul K_n^{MW} M_m)^{tr}(\xi) \to (\ul K_{n-1}^{MW}
M_m)^{tr}(\kappa(x), \Lambda_x^X)$ is surjective. Let $l/\kappa(x)$ be a finite extension.
Writing this as a sequence of simple extensions and lifting minimal polynomials
of generators (which remain irreducible by Gauss' lemma),
we can find a finite extension $L/K$ with $[L:K] = [l/\kappa(x)]$
such that if $f: \tilde{X} \to X$ is the integral closure of $X$ in $L$ then $l$ is one
of the residue fields of $\tilde{X}$. Note that $\tilde{X}$ is connected
($\mathcal{O}_{\tilde{X}} \subset L$ is a domain), essentially $k$-smooth (since
normal) and
finite over $X$ \cite[Theorem 2.39(d)]{liu2002algebraic} ($X$ is excellent
\cite[Corollaire 18.7.6]{EGAIV}),
so is in fact Nisnevich local (i.e. the spectrum of a Henselian
local ring) \cite[Tag 04GH]{stacks-project}. Let $\tilde{x}$ be the unique point above $x$, so that
$l = \kappa(\tilde{x})$.
From Theorem \ref{thm:comm-tr} and Lemma \ref{lemm:KnMmtr-bdry} we get the commutative diagram
\begin{equation*}
\begin{CD}
(\ul K_n^{MW} M_m)(k(\tilde{X}), \nu(f)) @>>> (\ul K_{n-1}^{MW} M_m)(\kappa(\tilde{x}), \nu(f) \otimes \Lambda_{\tilde{x}}^{\tilde{X}}) \\
@VV{Tr}V                               @VV{Tr}V \\
(\ul K_n^{MW} M_m)^{tr}(k({X})) @>>> (\ul K_{n-1}^{MW} M_m)^{tr}(\kappa(x), \Lambda_x^X). \\
\end{CD}
\end{equation*}
Observing that $\nu(f)$ is (non-canonically) trivial (and that $l$ was
arbitrary)
it is thus enough to show that $(\ul K_n^{MW} M_m)(\xi) \to \ul
K_{n-1}(\kappa(x), \Lambda_x^X) M_m(\kappa(x))$ is surjective (renaming $\tilde{X}$ to $X$).

Since $k$ is perfect $\kappa(x)/k$ is smooth. It is then not hard to see (using
the fact that $X$ is Henselian) that $x \to X$ admits a section, and so $s: M_m(X)
\to M_m(\kappa(x))$ is surjective. Also $\partial: \ul K_n^{MW}(\xi) \to \ul
K_{n-1}(\kappa(x), \Lambda_x^X)$ is surjective because $\ul K_*^{MW}$ is a homotopy
module. So for $\alpha \in \ul K_{n-1}^{MW}(\kappa(x), \Lambda_x^X)$ and $m \in
M_m(\kappa(x))$ we can find $\alpha' \in \ul K_n^{MW}(\xi)$ and $m' \in K_m(X)$
with $\partial(\alpha') = \alpha, s(m') = m$ and thus $\partial(\alpha' m') =
\alpha m$ by a (very) special case of Proposition \ref{prop:s-partial-formulas}
part (iii).

This concludes the proof.
\end{proof}

\begin{lemm*}[\ref{lemm:A1-diff-surjective}]
Let $K \in \mathcal{F}_k$.
The differential
\[ (\ul K_n^{MW} M_m)^{tr}(K(T)) \to \bigoplus_{x \in
(\Aone_K)^{(1)}} (\ul K_{n-1}^{MW} M_m)^{tr}(k(x), \Lambda_x^{\Aone})\] is surjective.
\end{lemm*}
\begin{proof}[Proof of Lemma \ref{lemm:A1-diff-surjective}.]
Fix $x_0 \in (\Aone_K)^{(1)}$. It suffices to prove that given $\alpha \in (\ul
K_{n-1}^{MW} M_m)^{tr}(\kappa(x_0), \Lambda_{x_0}^{\Aone})$ there exists $\beta \in (\ul K_n^{MW}
M_m)(K(T))$ such that $\partial_{x_0}(\beta) = \alpha$ and $\partial_y(\beta) =
0$ for $y \ne x_0$. We may suppose that $\alpha$ is obtained by transfer from a
finite extension $L/\kappa(x)$. Then $L/K$ is also finite.

Letting $f: \Aone_L \to \Aone_K$ be the canonical map, as
before we get a commutative diagram
\begin{equation*}
\begin{CD}
(\ul K_n^{MW} M_m)(L(T), \nu(f)) @>>> \bigoplus_{y \in (\Aone_L)^{(1)}} (\ul K_{n-1}^{MW} M_m)(k(y), \nu(f) \otimes \Lambda_y^{\Aone}) \\
   @VV{Tr}V                                @VV{Tr}V    \\
(\ul K_n^{MW} M_m)^{tr}(K(T)) @>>> \bigoplus_{x \in (\Aone_K)^{(1)}} (\ul K_{n-1}^{MW} M_m)^{tr}(k(x), \Lambda_x^{\Aone}). \\
\end{CD}
\end{equation*}
We note that there is at least one point $y_0 \in \Aone_L$ above $x_0 \in
\Aone_K$ with $\kappa(y_0) = \kappa(x_0)$ and that $\nu(f)$ is (non-canonically)
trivial. Hence it suffices to prove: given
$y_0 \in (\Aone_L)^{(1)}$ with $\kappa(y) = L$ and $\alpha \in (\ul K_{n-1}^{MW}
M_m)(\kappa(y_0), \Lambda_{y_0}^{\Aone})$ there exists $\beta \in (\ul K_n^{MW}
M_m)(L(T))$ with $\partial_{y_0}(\beta) = \alpha$ and $\partial_y(\beta) = 0$
for $y \ne y_0$.

We may assume that $\alpha = c m$ with $c \in \ul K_{n-1}^{MW}(\kappa(y_0),
\Lambda_{y_0}^{\Aone})$ and $m \in M_m(\kappa(y_0))$. Since $\ul K_*^{MW}$ is a
homotopy module there exists $c' \in \ul K_n^{MW}(L(T))$ with
$\partial_{y_0}(c') = c$ and $\partial_y(c') = 0$ for $y \ne y_0$. Moreover
since $\kappa(y_0) = \kappa(y)$ the natural map $M_m(L) \to M_m(\kappa(y))$ is
an isomorphism and we may choose $m' \in M_m(L)$ mapping to $m$. Then $m' \in
M_m(\mathcal{O}_{\Aone_L, y})$ for all $y$ and hence $\partial_{y}(c' m') =
\partial_y(c') s(m')$ for all $y$ (including $y_0$). Thus
$\partial_{y_0}(c'm') = cm$ and $\partial_y(c'm') = 0$ for $y \ne y_0$. This
concludes the proof.
\end{proof}

\bibliographystyle{plainc}
\bibliography{bibliography}

\begin{thebibliography}{10}

\bibitem{levine2015witt}
Alexey Ananyevskiy, Marc Levine, and Ivan Panin.
\newblock Witt sheaves and the $\eta$-inverted sphere spectrum.
\newblock {\em Journal of Topology}, 10(2):370--385, 2017.

\bibitem{bachmann-real-etale}
Tom Bachmann.
\newblock Motivic and Real Etale Stable Homotopy Theory.
\newblock 2016.
\newblock \href{https://arxiv.org/abs/1608.08855}{arXiv:1608.08855}.

\bibitem{bachmann-quadrics}
Tom Bachmann.
\newblock On the Invertibility of Motives of Affine Quadrics.
\newblock {\em Documenta Math.}, 22:363--395, 2017.
\newblock \href{https://arxiv.org/abs/1506.07377}{arXiv:1506.07377}.

\bibitem{beilinson1982faisceaux}
Alexander~A Beilinson, Joseph Bernstein, and Pierre Deligne.
\newblock Faisceaux pervers. Analysis and topology on singular spaces, I
  (Luminy, 1981), 5--171.
\newblock {\em Ast{\'e}risque}, 100, 1982.

\bibitem{bondarko-effectivity}
Mikhail~V. Bondarko.
\newblock On infinite effectivity of motivic spectra and the vanishing of their
  motives.
\newblock {\em arXiv preprint arXiv:1602.04477}, 2016.

\bibitem{cisinski2009local}
Denis-Charles Cisinski and Fr{\'e}d{\'e}ric D{\'e}glise.
\newblock Local and stable homological algebra in Grothendieck abelian
  categories.
\newblock {\em Homology, Homotopy and Applications}, 11(1):219--260, 2009.

\bibitem{cisinski2013etale}
Denis-Charles Cisinski and Fr{\'e}d{\'e}ric D{\'e}glise.
\newblock {\'E}tale Motives.
\newblock {\em Compositio Mathematica}, pages 1--111, 2013.

\bibitem{deglise-htpy-modules}
Fr{\'e}d{\'e}ric D{\'e}glise.
\newblock Orientable homotopy modules.
\newblock {\em American Journal of Mathematics}, 135(2):519--560, 2013.

\bibitem{delfs1991homology}
Hans Delfs.
\newblock {\em Homology of locally semialgebraic spaces}.
\newblock Lecture Notes in Mathematics. Springer-Verlag Berlin Heidelberg,
  1991.

\bibitem{elman-lum-2cohom}
Richard Elman and Christopher Lum.
\newblock On the cohomological 2-dimension of fields.
\newblock {\em Communications in Algebra}, 27(2):615--620, 1999.

\bibitem{gelfand-methods}
S.I. Gelfand and Y.~Manin.
\newblock {\em Methods of Homological Algebra}.
\newblock Springer Monographs in Mathematics. Springer Science \& Business
  Media, 2003.

\bibitem{EGAIV}
A.~{Grothendieck}.
\newblock {\'El\'ements de g\'eom\'etrie alg\'ebrique. IV: \'Etude locale des
  sch\'emas et des morphismes de sch\'emas. R\'edig\'e avec la colloboration de
  Jean Dieudonn\'e.}
\newblock {\em {Publ. Math., Inst. Hautes \'Etud. Sci.}}, 32:1--361, 1967.

\bibitem{hovey2001spectra}
Mark Hovey.
\newblock Spectra and symmetric spectra in general model categories.
\newblock {\em Journal of Pure and Applied Algebra}, 165(1):63--127, 2001.

\bibitem{hovey1997axiomatic}
Mark Hovey, John~Harold Palmieri, and Neil~P Strickland.
\newblock {\em Axiomatic stable homotopy theory}, volume 610.
\newblock American Mathematical Soc., 1997.

\bibitem{hoyois-algebraic-cobordism}
Marc Hoyois.
\newblock From algebraic cobordism to motivic cohomology.
\newblock {\em Journal f{\"u}r die reine und angewandte Mathematik (Crelles
  Journal)}, 2015(702):173--226, 2015.

\bibitem{HuPicard}
Po~Hu.
\newblock On the {P}icard group of the stable $\mathbb{A}^1$-homotopy category.
\newblock {\em Topology}, 44(3):609--640, 2005.

\bibitem{karoubi2015witt}
Max Karoubi, Marco Schlichting, and Charles Weibel.
\newblock The Witt group of real algebraic varieties.
\newblock {\em Journal of Topology}, 9(4):1257--1302, 2016.

\bibitem{wittrings}
M.~Knebusch and M.~Kolster.
\newblock {\em Wittrings}.
\newblock Aspects of Mathematics. Vieweg, 1982.

\bibitem{knebusch-bilinear}
Manfred Knebusch.
\newblock Symmetric bilinear forms over algebraic varieties.
\newblock In G.~Orzech, editor, {\em Conference on quadratic forms}, volume~46
  of {\em Queen's papers in pure and applied mathematics}, pages 103--283.
  Queens University, Kingston, Ontario, 1977.

\bibitem{levine-slice}
Marc Levine.
\newblock The slice filtration and Grothendieck-Witt groups.
\newblock {\em Pure and Applied Mathematics Quarterly}, 7(4), 2011.
\newblock Special Issue: In memory of Eckart Viehweg.

\bibitem{liu2002algebraic}
Qing Liu.
\newblock {\em Algebraic geometry and arithmetic curves}, volume~6.
\newblock Oxford University Press on Demand, 2002.

\bibitem{MayBurnsideRing}
J.P. May.
\newblock Picard Groups, Grothendieck Rings, and Burnside Rings of Categories.
\newblock {\em Advances in Mathematics}, 163(1):1 -- 16, 2001.

\bibitem{lecture-notes-mot-cohom}
Carlo Mazza, Vladimir Voevodsky, and Charles Weibel.
\newblock {\em Lecture notes on motivic cohomology}.
\newblock American Mathematical Soc., 2006.

\bibitem{milnor1973symmetric}
John~Willard Milnor and Dale Husemoller.
\newblock {\em Symmetric bilinear forms}, volume~60.
\newblock Springer, 1973.

\bibitem{morel-voevodskys-proof}
Fabien Morel.
\newblock Voevodsky's proof of {M}ilnor's conjecture.
\newblock {\em Bull. Amer. Math. Soc.}, 1998.

\bibitem{morel-trieste}
Fabien Morel.
\newblock An introduction to $\mathbb{A}^1$-homotopy theory.
\newblock {\em ICTP Trieste Lecture Note Ser. 15}, pages 357--441, 2003.

\bibitem{morel2004motivic-pi0}
Fabien Morel.
\newblock On the motivic $\pi_0$ of the sphere spectrum.
\newblock In {\em Axiomatic, enriched and motivic homotopy theory}, pages
  219--260. Springer, 2004.

\bibitem{morel2005stable}
Fabien Morel.
\newblock The stable $\mathbb{A}^1$-connectivity theorems.
\newblock {\em K-theory}, 35(1):1--68, 2005.

\bibitem{A1-alg-top}
Fabien Morel.
\newblock {\em $\mathbb{A}^1$-Algebraic Topology over a Field}.
\newblock Lecture Notes in Mathematics. Springer Berlin Heidelberg, 2012.

\bibitem{A1-homotopy-theory}
Fabien Morel and Vladimir Voevodsky.
\newblock $\mathbb{A}^1$-homotopy theory of schemes.
\newblock {\em Publications Mathématiques de l'Institut des Hautes Études
  Scientifiques}, 90(1):45--143, 1999.

\bibitem{Neeman1992}
Amnon Neeman.
\newblock The connection between the $K$-theory localization theorem of
  Thomason, Trobaugh and Yao and the smashing subcategories of Bousfield and
  Ravenel.
\newblock {\em Annales scientifiques de l'École Normale Supérieure},
  25(5):547--566, 1992.

\bibitem{panin-witt-purity}
Manuel Ojanguren and Ivan Panin.
\newblock A purity theorem for the Witt group.
\newblock {\em Annales Scientifiques de l’École Normale Supérieure},
  32(1):71 -- 86, 1999.

\bibitem{rigidity-in-motivic-homotopy-theory}
Oliver R{\"o}ndigs and Paul~Arne {\O}stv{\ae}r.
\newblock Rigidity in motivic homotopy theory.
\newblock {\em Mathematische Annalen}, 341(3):651--675, 2008.

\bibitem{RondigsModules}
Oliver Röndigs and Paul~Arne Østvær.
\newblock Modules over motivic cohomology.
\newblock {\em Advances in Mathematics}, 219(2):689 -- 727, 2008.

\bibitem{scharlau1972quadratic}
Winfried Scharlau.
\newblock Quadratic reciprocity laws.
\newblock {\em Journal of Number Theory}, 4(1):78--97, 1972.

\bibitem{real-and-etale-cohomology}
Claus Scheiderer.
\newblock {\em Real and Etale Cohomology}, volume 1588 of {\em Lecture Notes in
  Mathematics}.
\newblock Springer, Berlin, 1994.

\bibitem{shatz1972profinite}
S.S. Shatz.
\newblock {\em Profinite Groups, Arithmetic, and Geometry}.
\newblock Annals of mathematics studies. Princeton University Press, 1972.

\bibitem{stacks-project}
The {Stacks Project Authors}.
\newblock {\itshape Stacks Project}.
\newblock \url{http://stacks.math.columbia.edu}, 2017.

\bibitem{voevodsky-slice-filtration}
V.~Voevodsky.
\newblock Open Problems in the Motivic Stable Homotopy Theory , I.
\newblock In {\em International Press Conference on Motives, Polylogarithms and
  Hodge Theory}. International Press, 2002.

\end{thebibliography}

\end{document}